\documentclass[11pt,reqno]{amsart}
\usepackage{enumerate,amsfonts,amsmath,amsthm,amssymb,latexsym,verbatim,amscd,mathrsfs,pb-diagram, graphics}
\usepackage{hyperref}
\usepackage[usenames,dvipsnames]{color}
\usepackage{setspace}
\usepackage{etoolbox}
\allowdisplaybreaks

\usepackage{geometry}
\theoremstyle{definition}
\newtheorem{thm}{Theorem}[section]
\newtheorem{prop}[thm]{Proposition}
\newtheorem{cor}[thm]{Corollary}
\newtheorem{lem}[thm]{Lemma}

\newcommand{\lie}[1]{\mathfrak{#1}}
\newcommand{\cal}[1]{\mathcal{#1}}

\newtoggle{details}
\iftoggle{details}{\newcommand{\details}[1]{{\color{blue}#1}}}{ \newcommand{\details}[1]{}}

\begin{document}
 
\title[An Integral Basis for the Onsager Algebra]{An Integral Basis for the Universal Enveloping Algebra of the Onsager Algebra}
  
\author{Angelo Bianchi}
\email{acbianchi@unifesp.br}
\address{Universidade Federal de S\~ao Paulo }
\author{Samuel Chamberlin}
\email{samuel.chamberlin@park.edu}
\address{Park University, 8700 River Park Drive Box No. 30, Parkville, MO 64152}

\begin{abstract}
We construct an integral form for the universal enveloping algebra of the \textit{Onsager algebra} and an explicit integral basis for this integral form. We also formulate straightening identities among some products of basis elements.

\end{abstract}

\maketitle

MSC2020: Primary 17B65; Secondary 17B05

\tableofcontents

\section*{Introduction}

The focus of this paper is the Onsager algebra, which took on the name of the Nobel Prize winning Lars Onsager (1903-1976). This algebra appeared for the first time in Onsager's solution of the two-dimensional Ising model in a zero magnetic field,  \cite{O44}. In this work, two non-commuting matrices appeared related to the transfer matrix associated to this model. By analyzing the structure of the algebra generated by these matrices, Onsager derived a complex Lie algebra. There are a few different realizations of the Onsager algebra. We refer the reader to \cite{el} for a survey with details. The realization used here is as an equivariant map algebra. Equivariant map algebras have the form $\lie g_\mathcal A:=(\mathfrak{g}\otimes_{\mathbb C}\mathcal{A})^\Gamma$, where $\mathfrak{g}$ is a finite-dimensional complex simple Lie algebra, $\mathcal{A}$ is a commutative, associative, unital $\mathbb C$-algebra, and $\Gamma$ is a group acting on $\lie g$ and $\cal A$ (and hence diagonally on $\lie g_\mathcal A$). 

We highlight the work of Roan \cite{R91} who first investigated the relationship between the Onsager algebra and the $\mathfrak{sl}_2$-loop algebra, folowed by Benkart and  Terwilliger \cite{BT07} and  Hartwig and Terwilliger \cite{HT07}. Since then, a few authors have written papers on this topic, see for instance \cite{DR00,E07}.

Suitable integral forms and bases for the universal enveloping algebras of the complex simple finite-dimensional Lie algebras were formulated by Kostant in 1966 (cf. \cite{K}), after Chevalley investigated integral forms for the classical Lie algebras in 1955. This led to the
construction of the classical Chevalley groups. In 1978, Garland formulated useful integral forms and bases for the universal enveloping algebras of the (untwisted) loop algebras, \cite{G}. Mitzman extended Garland's integral bases to the universal enveloping algebras of the twisted loop algebras in 1983, \cite{M}. The second author formulated suitable integral forms and bases for the universal enveloping algebras of the (untwisted) map algebras and (untwisted) map superalgebras, \cite{BC,C}.

Once one has access to suitable integral forms and bases for a particular Lie algebra one can study its representation theory in positive characteristic via its hyperalgebra. This was done for the (untwisted) hyper-loop algebras by Jakeli\'c and Moura using Garland's integral bases, \cite{JM1}. The first author and Moura extended this to the twisted hyper-loop algebras using Mitzman's integral bases, \cite{BM}. The authors extended these results to the (untwisted) hyper-multiloop, hyper-multicurrent, and hyper-map algebras this year using the integral bases formulated by the first author, \cite{BiC1,BiC2}.

When we view the Onsager algebra as an equivariant map algebra, the defining automorphism (the standard Chevalley involution) is not an isomorphism of the Dynkin diagram. No integral forms for any equivariant map algebras have been formulated in such a case in the literature.

Section \ref{pre} is dedicated to the algebraic preliminaries, including a review of the Onsager algebra. In Section \ref{main}, we construct an integral form and an integral basis for the Onsager algebra, the main result of the paper is stated as Theorem \ref{thm}. Section \ref{proof} contains all relevant proofs for the main result of the paper, which includes the necessary straightening identities. 

\

\noindent\textbf{Note on the ArXiv version}  For the interested reader, the .tex file of the ArXiv version of this paper includes some details of straightforward proofs omitted in the .pdf file. These details can be displayed by activating the \texttt{details} toggle in the .tex file and recompiling it.

\section{Preliminaries}\label{pre}

Throughout this work we denote by $\mathbb C$, $\mathbb Z$, $\mathbb Z_+$ and $\mathbb N$ the sets of complex numbers, integers, nonnegative integers and positive integers, respectively. 

\

If $\mathcal{A}$ is an  $\mathbb C$-algebra, an \textbf{integral form} $\mathcal{A}_{\mathbb Z}$ of $\mathcal{A}$ is a ${\mathbb Z}$-algebra such that $\mathcal{A}_{\mathbb Z}\otimes_{\mathbb Z}\mathbb C=\mathcal{A}$. An \textbf{integral basis} for $\mathcal{A}$ is a ${\mathbb Z}$-basis for $\mathcal{A}_{\mathbb Z}$.

\subsection{The \texorpdfstring{$\lie{sl}_2$}{}-algebra, its loop algebra, and a Chevalley involution} \label{sl2}

Let $\lie{sl}_2$ be the set of complex traceless order two matrices with Lie $\mathbb C$-algebra structure with bracket given by the commutator. Recall the standard basis of $\mathfrak{sl}_2$,
\begin{align*}
    \left\{ x^+ = \left( \begin{array}{cc} 0 & 1 \\ 0 & 0 \end{array} \right), \ h = \left( \begin{array}{cc} 1 & 0 \\ 0 & -1 \end{array} \right), \ x^- = \left( \begin{array}{cc} 0 & 0 \\ 1 & 0 \end{array} \right) \right\},
\end{align*}
satisfying $[x^+,x^-] = h$, $[h,x^+] = 2x^+$, and $[h,x^-] = -2x^-$.

In what follows an unadorned tensor means a tensor product over the field $\mathbb C$. Let $\mathbb C[t,t^{-1}] $ be the set of complex Laurent polynomials in one variable. Then the (untwisted) loop algebra of $\lie{sl}_2$ is the Lie algebra $\widetilde{\lie{sl}_2}:=\lie{sl}_2\otimes \mathbb C[t,t^{-1}] $ with Lie bracket given by bilinearly extending the bracket 
$$[z\otimes f,z'\otimes g]=[z,z']\otimes fg,$$
where $z,z'\in\lie{sl}_2$ and $f,g\in\mathbb C [t,t^{-1}]$.

 The \emph{Chevalley involution} on $\mathfrak{sl}_2$ is defined by the map $$M \mapsto \overline{M} := \left( \begin{matrix} 0 & 1 \\ 1 & 0 \end{matrix} \right) M \left(\begin{matrix} 0 & 1 \\ 1 & 0 \end{matrix}\right) $$ for all $M \in \mathfrak{sl}_2$. Therefore, under this involution we have
$$x^+ \mapsto   x^-, \quad x^- \mapsto x^+, \quad  h \mapsto -h.$$
Notice that the Chevalley involution is an involution defined as an automorphism of order two.

The Chevalley involution on $\mathfrak{sl}_2$ induces an involution on $\widetilde{\mathfrak{sl}_2}$, also called the \textit{Chevalley involution} on $\widetilde{\mathfrak{sl}_2}$, and denoted by $\omega$ where
$$ \omega(g\otimes f(t)) := \overline{g}\otimes f(t^{-1}) , \quad \text{for } f(t) \in \mathbb C[t,t^{-1}],\ g \in \mathfrak{sl}_2.$$

\subsection{Universal enveloping algebras} 

For a Lie algebra $\lie a$, we denote by $U(\lie a)$ the corresponding universal enveloping algebra of the Lie algebra $\lie a$. Given $u\in U(\lie a)$ and $k\in{\mathbb Z}$ define the divided powers and binomials of $u$, respectively, as follows: $u^{(k)}=\binom{u}{k}=0$, if $k<0$, and
$$u^{(k)}:=\frac{u^k}{k!}\textnormal{ and }\binom{u}{k}:=\frac{u(u-1)\dots(u-k+1)}{k!},$$
if $k\geq0$.

Define $T^0(\lie a):=\mathbb C$, and for all $j\geq1$, define $T^j(\lie a):=\lie a^{\otimes j}$, $T(\lie a):=\bigoplus_{j=0}^\infty T^j(\lie a)$, and $T_j(\lie a):=\bigoplus_{k=0}^jT^k(\lie a)$. Then, set $U_j(\lie a)$ to be the image of $T_j(\lie a)$ under the canonical surjection $T(\lie a)\to U(\lie a)$, and for any $u\in U(\lie a)$ \emph{define the degree of $u$} by $$\deg u:=\min_{j}\{u\in U_j(\lie a)\}.$$

\subsection{The Onsager Algebra and one of its historical realizations}

We begin by with the original definition of the Onsager algebra and its relation to $\widetilde{\lie{sl_2}}$. We use the notation in \cite{el} instead of the original notation in \cite{O44}. 

\ 

The Onsager algebra, denoted $\mathcal{O}$, is the (nonassociative) $\mathbb C$-algebra with basis $\{A_m, G_l ~ | ~ m \in \mathbb{Z}, l \in \mathbb{N}_+\}$ and antisymmetric product given by:
\begin{align*}
	[A_l,A_m] &= 2G_{l-m}, \text{ for } l > m,  \\
  [G_l,A_m] &= A_{m+l} - A_{m-l},   \\
  [G_l,G_m] &= 0.
\end{align*}

\

By considering the $\mathbb C$-linear map $\gamma:\mathcal{O} \rightarrow \widetilde{\lie{sl}_2}$ defined by 
\begin{align}\label{onsager1}
A_m\longmapsto x^+\otimes t^m + x^-\otimes t^{-m}\qquad\qquad G_l\longmapsto\frac{1}{2}\left(h\otimes (t^l-t^{-l})\right)    
\end{align}
for $m\in\mathbb{Z}$ and $l\in\mathbb{N}$, one can see that the Onsager algebra is isomorphic to the Lie subalgebra of $\widetilde{\lie{sl}_2}$ fixed by the Chevalley involution $\omega$ introduced in the Subsection \ref{sl2} (see details in \cite{R91}).

\subsection{The Onsager Algebra as an equivariant map algebra.}

Following the language of equivariant map algebras as in \cite{NSS2}, let $\Gamma:=\langle\sigma\rangle\cong{\mathbb Z}_2$ be a group of order two. We act on $\lie{sl}_2$ by the involution: 
\begin{align*}
    \sigma\cdot(x^\pm)=-x^\mp\\
    \sigma\cdot(h)=-h,
\end{align*}
and on $\mathbb C[t, t^{-1}]$ by
$$\sigma\cdot(t) = t^{-1}.$$
This induces the diagonal action of $\Gamma$ on $\widetilde{\lie{sl}_2}$ given by  
$$\sigma\cdot(g\otimes f)=\sigma\cdot(g)\otimes\sigma\cdot(f),$$ for all $g\in\lie{sl}_2$ and all $f\in \mathbb C[t,t^{-1}] $.
Therefore, the Onsager algebra $\mathcal O$ is also isomorphic to the Lie subalgebra of $\widetilde{\lie{sl}_2}$ which is fixed by this action of $\Gamma$, according to \cite[Remark 3.11 and Lemma 3.3]{NSS}.

\section{A better realization of the Onsager Algebra \texorpdfstring{$\mathcal O$}{} } \label{main}

The realization of the Onsager algebra in \eqref{onsager1} is not suitable to construct an integral form for $\mathcal O$ due to the intrinsic difficulty to deal with brackets between elements in \eqref{onsager1} and, hence, to produce useful identities and  $\mathbb Z$-subalgebras in $U(\mathcal O)$. The following realization of $\mathcal O$ allows straightening identities and the construction of  suitable subalgebras, which then gives rise to a ``triangular type decomposition'' and an integral form for $U(\mathcal O)$. 
We start with the following basis for $\lie{sl}_2$ and then build the appropriate basis for $\mathcal O$ from this one.

We define the following elements in $\lie {sl}_2$: $$h^\Gamma:=-i\left(x^+-x^-\right), \quad x_+^\Gamma:=\frac{1}{2}\left(x^++x^--ih\right), \quad \text{ and } \quad x_-^\Gamma:=\frac{1}{2}\left(x^++x^-+ih\right).$$ 
Then, $\mathrm{span}\left\{h^\Gamma\right\}\cong\mathbb C$, $\lie{h}^*_\Gamma\cong\mathbb C$,  $\mathrm{span}\{x_-^\Gamma,h^\Gamma,x_+^\Gamma\}\cong\lie{sl}_2$, and $\sigma\left(x_\pm^\Gamma\right)=-x_\pm^\Gamma$.

Given $j,k,l\in{\mathbb Z}_+$, define
\begin{align*}
    h_k^\Gamma:=h^\Gamma\otimes\left(t^k+t^{-k}\right),\quad  
    x_{j}^{\Gamma,+}:=x_+^\Gamma\otimes\left(t^j-t^{-j}\right), \quad \text{ and } \quad x_{l}^{\Gamma,-}:=x_-^\Gamma\otimes\left(t^l-t^{-l}\right),
\end{align*}
and note that $h_{-k}^\Gamma=h_k^\Gamma$, $x_{-j}^{\Gamma,\pm}=-x_{j}^{\Gamma,\pm}$, and $x_{0}^{\Gamma,\pm}=0$. 

Therefore, by defining  
$$u_j:=\left(x^+-x^-\right)\otimes\left(t^j+t^{-j}\right), \quad v_k:=\left(x^++x^-\right)\otimes\left(t^k-t^{-k}\right), \quad \text{ and } \quad w_l:=h\otimes\left(t^l-t^{-l}\right),$$
for all $j,k,l\in{\mathbb Z}$, then, using a similar argument to that in \cite[Proposition 3.2.2]{el}, we see that the set $\left\{u_j,v_k,w_l\ |\ j\in{\mathbb Z}_+,\ k,l\in\mathbb N\right\}$ is a $\mathbb C$-basis for $\cal O$ and, hence, the relations $$u_j=ih_j^\Gamma, \quad  v_k=x_k^{\Gamma,+}+x_k^{\Gamma,-},\text{ and}\quad w_l=-i\left(x_l^{\Gamma,-}-x_l^{\Gamma,+}\right)$$ demonstrate that the set $\mathcal B:=\left\{x_{l}^{\Gamma,-},h_k^\Gamma,x_{j}^{\Gamma,+}\ |\ k\in{\mathbb Z}_+,\ j,l\in\mathbb N\right\}$ spans $\cal O$. Furthermore, since $\cal B$ is clearly linearly independent it is a basis for $\cal O$.

 We can define an ordering, $\leq$, on $\cal B$ such that $x_{l}^{\Gamma,-}\leq h_{k}^{\Gamma}\leq x_{j}^{\Gamma,+}$ whenever $k\in\mathbb Z_+$ and $j,l\in\mathbb N$, $x_{l}^{\Gamma,\pm}\leq x_{j}^{\Gamma,\pm}$ whenever $l,j\in\mathbb N$ with $l\leq j$, and $h_{k}^{\Gamma}\leq h_{m}^{\Gamma}$ whenever $k,m\in\mathbb Z_+$ with $k\leq m$. Note that
\begin{align*}
    \left[x_{j}^{\Gamma,+},x_{l}^{\Gamma,-}\right]
    &=h^\Gamma_{j+l}-h^\Gamma_{j-l}=h^\Gamma_{j+l}-h^\Gamma_{|j-l|};\\
    \left[h_k^\Gamma,x_{j}^{\Gamma,+}\right]
    &=2\left(x^{\Gamma,+}_{j+k}+x^{\Gamma,+}_{j-k}\right)=2\left(x^{\Gamma,+}_{j+k}-x^{\Gamma,+}_{k-j}\right);\\
    \left[h_k^\Gamma,x_{l}^{\Gamma,-}\right]
    &=-2\left(x^{\Gamma,-}_{l+k}-x^{\Gamma,-}_{k-l}\right)=-2\left(x^{\Gamma,-}_{l+k}+x^{\Gamma,-}_{l-k}\right).
\end{align*}

\subsection{An integral form and integral basis for \texorpdfstring{$U(\mathcal O)$}{}.}

    Define $U_{\mathbb Z}(\mathcal O)$ to be the ${\mathbb Z}$-subalgebra of $U(\mathcal O)$ generated by
    $$\left\{\left(x_k^{\Gamma,\pm}\right)^{(s)}\ \bigg|\ k\in\mathbb N,\ s\in{\mathbb Z}_+\right\},$$
    $U_{\mathbb Z}^+(\mathcal O)$ and $U_{\mathbb Z}^-(\mathcal O)$ to be the ${\mathbb Z}$-subalgebras of $U_{\mathbb Z}(\mathcal O)$ generated respectively by
$$\left\{\left(x_k^{\Gamma,+}\right)^{(s)}\ \bigg|\ k\in\mathbb N,\ s\in{\mathbb Z}_+\right\}\textnormal{ and }\left\{\left(x_k^{\Gamma,-}\right)^{(s)}\ \bigg|\ k\in\mathbb N,\ s\in{\mathbb Z}_+\right\},$$
and $U_{\mathbb Z}^0(\mathcal O)=U\left(\left(\lie g^0\otimes\mathbb C [t,t^{-1}]\right)^\Gamma\right)\bigcap U_{\mathbb Z}(\mathcal O)$.

We need the following definitions in order to give our integral basis and to state our straightening identities.

Given $j,l\in\mathbb N$, define $\Lambda^\Gamma_{j,l,1}=-\left(h^\Gamma_{j+l}-h^\Gamma_{j-l}\right)$ and $D^{\Gamma,\pm}_{u,1}(j,l)\in U(\mathcal O)$, for $u\in{\mathbb Z}_+$, recursively as follows:
\begin{align*}
    D^{\Gamma,+}_{0,1}(j,l)
    &=\left(x_{j}^{\Gamma,+}\right);\\
    D^{\Gamma,-}_{0,1}(j,l)
    &=\left(x_{l}^{\Gamma,-}\right);\\
    D^{\Gamma,\pm}_{u,1}(j,l)
    &=\pm\frac{1}{2}\left[D^{\Gamma,\pm}_{u-1,1}(j,l),\Lambda_{j,l,1}^\Gamma\right].
\end{align*}

Furthermore, for $j,k,l\in\mathbb N$, set $p^{\Gamma}_k(j,l):=\left[x_{j}^{\Gamma,+},D^{\Gamma,-}_{k-1,1}(j,l)\right]$ and for $j,l\in\mathbb N$, and $k\in{\mathbb Z}$, define $\Lambda^\Gamma_{j,l,k}$ by equating coefficients in the following formal series: 
    $$\Lambda_{j,l}^\Gamma(u):=\sum_{r=0}^\infty\Lambda_{j,l,r}^\Gamma u^r=\exp\left(-\sum_{s=1}^\infty\frac{p_s^\Gamma(j,l)}{s}u^s\right).$$
In particular, $\Lambda^\Gamma_{j,l,k}=0$ for $k<0$, and $\Lambda^\Gamma_{j,l,0}=1$.

\begin{prop}
    Given $j,k,l\in\mathbb N$,
    \begin{equation*}
        \Lambda^\Gamma_{j,l,k}=-\frac{1}{k}\sum_{i=1}^kp_i^\Gamma(j,l)\Lambda^\Gamma_{j,l,k-i}.
    \end{equation*}
\end{prop}
\begin{proof}
    We use an argument similar to that in \cite[Lemma 3.2]{CP}. Differentiating both sides of 
    $$\log\Lambda^\Gamma_{j,l}(u)=-\sum_{s=1}^\infty\frac{p_s^\Gamma(j,l)}{s}u^s$$
    with respect to $u$ and then multiplying by $u$ gives
    $$u\frac{\left(\Lambda^\Gamma_{j,l}\right)'(u)}{\Lambda^\Gamma_{j,l}(u)}=-\sum_{s=1}^\infty p_s^\Gamma(j,l)u^s.$$
    Multiplying both sides of this equation by $\Lambda^\Gamma_{j,l}(u)$ gives
    $$\sum_{k=1}^\infty k\Lambda_{j,l,k}^\Gamma u^k=-\sum_{s=1}^\infty p_s^\Gamma(j,l)u^s\Lambda_{j,l}^\Gamma(u).$$
    Expanding $\Lambda_{j,l}^\Gamma(u)$ and the product on the right side and equating coefficients gives the result.
\end{proof}

Given a PBW monomial with respect to the order on $\mathcal B$, we construct an ordered monomial in the elements of the set
$$M:=\left\{\left(x^{\Gamma,+}_j\right)^{(r)},\Lambda^\Gamma_{j,l,k},\left(x^{\Gamma,-}_l\right)^{(s)}\ \bigg|\ j,l\in\mathbb N,\ k,r,s\in{\mathbb Z}_+\right\}$$
via the correspondence 
$$\left(x^{\Gamma,\pm}_j\right)^k\leftrightarrow\left(x^{\Gamma,\pm}_j\right)^{(k)}\quad\text{ and }\quad\left(p_1^\Gamma(l,m)\right)^{r}\leftrightarrow\Lambda^\Gamma_{l,m,r}.$$

\

The main goal of this paper is to prove the following theorem, whose proof is in Section \ref{proof} (see \cite{BC,C,G,H,M} for analogs in different settings).

\begin{thm}\label{thm}
    The subalgebra $U_{\mathbb Z}(\mathcal O)$ is a free ${\mathbb Z}$-module and the set of ordered monomials constructed from $M$ is a ${\mathbb Z}$ basis of $U_{\mathbb Z}(\mathcal O)$.
\end{thm}

This theorem implies the following:
\begin{align*}
    \mathbb C\otimes_{\mathbb Z} U_{\mathbb Z}(\mathcal O)&\cong U(\mathcal O),\\
    \mathbb C\otimes_{\mathbb Z} U_{\mathbb Z}^\pm(\mathcal O)&\cong U\left(\mathrm{span}\left\{x_j^{\Gamma,\pm}\ |\ j\in\mathbb N\right\}\right),\\
    \mathbb C\otimes_{\mathbb Z} U_{\mathbb Z}^0(\mathcal O)&\cong U\left(\mathrm{span}\left\{h_k^{\Gamma}\ |\ k\in{\mathbb Z}_+\right\}\right).\\
\end{align*}
In particular, $U_{\mathbb Z}(\mathcal O)$ is an \emph{integral form} of $U(\mathcal O)$.

\subsection{Identities for \texorpdfstring{$D^{\Gamma,\pm}_{u,v}(j,l)$}{}.}
Define $D^{\Gamma,\pm}_{u,v}(j,l)\in U(\mathcal O)$ recursively as follows:
\begin{align*}
    D^{\Gamma,\pm}_{u,v}(j,l)
    &=0\textnormal{ if }v<0;\\
    D^{\Gamma,\pm}_{u,0}(j,l)
    &=\delta_{u,0};\\
    D^{\Gamma,\pm}_{u,v}(j,l)
    &=\frac{1}{v}\sum_{i=0}^uD^{\Gamma,\pm}_{i,1}(j,l)D^{\Gamma,\pm}_{u-i,v-1}(j,l).
\end{align*}

The following proposition gives the $D^{\Gamma,\pm}_{u,v}(j,l)$ as the coefficients of a power series in the indeterminate $w$. The proof of the first equation is done by induction on $v$ and we omit the details. The second equation follows by the Multinomial Theorem.

\begin{prop}\label{Duv}
    For all $j,l,u,v\in\mathbb N$
    \begin{align*}
        D^{\Gamma,\pm}_{u,v}(j,l)
        &=\sum_{\substack{k_0,\ldots,k_u\in{\mathbb Z}_+\\k_0+\dots k_u=v\\k_1+2k_2+\dots+uk_u=u}}\left(D^{\Gamma,\pm}_{0,1}(j,l)\right)^{(k_0)}\dots\left(D^{\Gamma,\pm}_{u,1}(j,l)\right)^{(k_u)}\\
        &=\left(\left(\sum_{m\geq0}D^{\Gamma,\pm}_{m,1}(j,l)w^{m+1}\right)^{(v)}\right)_{u+v} 
    \end{align*}
    \hfill\qedsymbol
\end{prop}

It can be easily shown by induction that $D^{\Gamma,+}_{0,v}(j,l)=\left(x_{j}^{\Gamma,+}\right)^{(v)}$ and $D^{\Gamma,-}_{0,v}(j,l)=\left(x_{l}^{\Gamma,-}\right)^{(v)}$.

A straightforward calculation shows that
\begin{align}
    \left[\left(x_{k}^{\Gamma,+}\right),\Lambda_{j,l,1}^\Gamma\right]
    &=2\left(x_{k+j+l}^{\Gamma,+}\right)+2\left(x_{k-j-l}^{\Gamma,+}\right)-2\left(x_{k+j-l}^{\Gamma,+}\right)-2\left(x_{k-j+l}^{\Gamma,+}\right).\label{xk+Lambda_1}
\end{align}

The following proposition gives the $D^{\Gamma,\pm}_{u,1}(j,l)$ as ${\mathbb Z}$-linear combinations of the $\left(x^{\Gamma,\pm}_{n}\right)$.

\begin{prop}\label{Du1}
    For all $u\in{\mathbb Z}_+$ and $j,l\in\mathbb N$
    \begin{align}
        D^{\Gamma,+}_{u,1}(j,l)
        &=\sum_{k=0}^{\left\lfloor\frac{u-1}{2}\right\rfloor}\sum_{i=0}^{u+1}(-1)^{k+i}\binom{u}{k}\binom{u+1}{i}\left(x_{(u+1-2i)j+(u-2k)l}^{\Gamma,+}\right)\nonumber\\
        &+((u+1)\mod2)\sum_{i=0}^{\frac{u}{2}}(-1)^{\frac{u}{2}+i}\binom{u}{\frac{u}{2}}\binom{u+1}{i}\left(x_{(u+1-2i)j}^{\Gamma,+}\right)\label{D+u1};\\
        D^{\Gamma,-}_{u,1}(j,l)
        &=\sum_{k=0}^{\left\lfloor\frac{u-1}{2}\right\rfloor}\sum_{i=0}^{u+1}(-1)^{k+i}\binom{u}{k}\binom{u+1}{i}\left(x_{(u+1-2i)l+(u-2k)j}^{\Gamma,-}\right)\nonumber\\
        &+((u+1)\mod2)\sum_{i=0}^{\frac{u}{2}}(-1)^{\frac{u}{2}+i}\binom{u}{\frac{u}{2}}\binom{u+1}{i}\left(x_{(u+1-2i)l}^{\Gamma,-}\right)\label{D-u1}.
    \end{align}
\end{prop} 

The proof of \eqref{D+u1} has been omitted to shorten the text due to its straightforwardness.  \eqref{D-u1} follows from \eqref{D+u1} by applying the automorphism $\tau:\cal O\to\cal O$ given by $\tau\left(x^{\Gamma,\pm}_j\right)=x^{\Gamma,\mp}_j$ and $\tau\left(h^{\Gamma}_k\right)=-h^\Gamma_k$ and switching $j$ and $l$. 

\subsection{Straightening Identities}
In this section we state all necessary straightening identities.

Given $u\in{\mathbb Z}_+$ and $j,k,m\in\mathbb N$ define
\begin{equation*}
    D^{\Gamma,\pm}_u(j,k,m)=\sum_{n=0}^u\sum_{v=0}^u(-1)^{n+v}\binom{u}{n}\binom{u}{v}\left(x_{j+(u-2n)k+(u-2v)m}^{\Gamma,\pm}\right).
\end{equation*}

\begin{prop}\label{straightening}
    For all $j,l,k,m\in\mathbb N$ and $n,r,s\in{\mathbb Z}_+$
    \begin{align}
        \Lambda^\Gamma_{j,l,r}\Lambda^\Gamma_{k,m,n}
        &=\Lambda^\Gamma_{k,m,n}\Lambda^\Gamma_{j,l,r}\\
        \left(x_j^{\Gamma,\pm}\right)^{(r)}\left(x_j^{\Gamma,\pm}\right)^{(s)}
        &=\binom{r+s}{s}\left(x_j^{\Gamma,\pm}\right)^{(r+s)}\\
        \left(x_j^{\Gamma,+}\right)^{(r)}\left(x_l^{\Gamma,-}\right)^{(s)}
        &=\sum_{\substack{m,n,q\in{\mathbb Z}_+\\m+n+q\leq\min\{r,s\}}}(-1)^{m+n+q}D^{\Gamma,-}_{m,s-m-n-q}(j,l)\Lambda^\Gamma_{j,l,n}D^{\Gamma,+}_{q,r-m-n-q}(j,l)\\
        \left(x_j^{\Gamma,+}\right)^{(r)}\Lambda^\Gamma_{k,m,n}
        &=\sum_{i=0}^n\Lambda_{k,m,n-i}^\Gamma\prod_{\substack{v_0\ldots v_u\in{\mathbb Z}_+\\\sum v_u=r\\\sum uv_u=i}}\left((u+1)D^{\Gamma,+}_{u}(j,k,m)\right)^{(v_u)}\\
        \Lambda^\Gamma_{k,m,n}\left(x_l^{\Gamma,-}\right)^{(s)}
        &=\sum_{i=0}^n\prod_{\substack{v_0\ldots v_u\in{\mathbb Z}_+\\\sum v_u=s\\\sum uv_u=i}}\left((u+1)D^{\Gamma,-}_{u}(l,k,m)\right)^{(v_u)}\Lambda_{k,m,n-i}^\Gamma
    \end{align}
\end{prop}
(5) and (6) are clear. Before proving (7)--(9)  we will state and prove a necessary corollary and prove Theorem \ref{thm}.

\begin{cor}
    For all $j,l\in\mathbb N$ and $u,v\in{\mathbb Z}_+$
    \begin{align*}
        (i)_u&\ \ D^{\Gamma,\pm}_{u,v}(j,l)\in U_{\mathbb Z}(\mathcal O);\\
        (ii)_u&\ \ \Lambda^\Gamma_{j,l,u}\in U_{\mathbb Z}(\mathcal O).
    \end{align*}
\end{cor}
\begin{proof}
    We will prove both statements simultaneously by induction on $u$ according to the scheme
    $$(ii)^u\Rightarrow (i)^u\Rightarrow (ii)^{u+1}$$
    where $(i)^u$ is the statement that $(i)_m$ holds for all $m\leq u$ and similarly for $(ii)^u$. $(i)_0$, $(ii)_0$ and $(ii)_1$ are all clearly true. Assume that $(ii)^u$ and $(i)^{u-1}$ hold for some $u\in{\mathbb Z}_+$. Then, by Proposition \ref{straightening}(7) we have
    \begin{align*}
        \left(x_j^{\Gamma,+}\right)^{(u+v)}\left(x_l^{\Gamma,-}\right)^{(u)}
        &=\sum_{\substack{m,n,q\in{\mathbb Z}_+\\m+n+q\leq u}}(-1)^{m+n+q}D^{\Gamma,-}_{m,u-m-n-q}(j,l)\Lambda^\Gamma_{j,l,n}D^{\Gamma,+}_{q,u+v-m-n-q}(j,l)\\
        &=(-1)^{u}D^{\Gamma,+}_{u,v}(j,l)+\sum_{\substack{m,n,q\in{\mathbb Z}_+\\m+n+q\leq u\\m,q<u}}(-1)^{m+n+q}D^{\Gamma,-}_{m,u-m-n-q}(j,l)\Lambda^\Gamma_{j,l,n}D^{\Gamma,+}_{q,u+v-m-n-q}(j,l)
    \end{align*}
    The summation is in $U_{\mathbb Z}(\mathcal O)$ by the induction hypothesis and $(ii)^u$. Therefore, $(i)^u$ holds. 
    
    Again, by Proposition \ref{straightening}(7), we have
    \begin{align*}
        \left(x_j^{\Gamma,+}\right)^{(u+1)}\left(x_l^{\Gamma,-}\right)^{(u+1)}
        &=\sum_{\substack{m,n,q\in{\mathbb Z}_+\\m+n+q\leq u+1}}(-1)^{m+n+q}D^{\Gamma,-}_{m,u+1-m-n-q}(j,l)\Lambda^\Gamma_{j,l,n}D^{\Gamma,+}_{q,u+1-m-n-q}(j,l)\\
        &=(-1)^{u+1}\Lambda^\Gamma_{j,l,u+1}\\
        &+\sum_{\substack{m,n,q\in{\mathbb Z}_+\\m+n+q\leq u+1\\n<u+1}}(-1)^{m+n+q}D^{\Gamma,-}_{m,u+1-m-n-q}(j,l)\Lambda^\Gamma_{j,l,n}D^{\Gamma,+}_{q,u+1-m-n-q}(j,l)\\
    \end{align*}
    The summation is in $U_{\mathbb Z}(\mathcal O)$ by the induction hypothesis and $(i)^u$. Therefore, $(ii)^{u+1}$ holds.
\end{proof}

\section{Proof of the main result }\label{proof}

\subsection{Proof of Theorem \ref{thm}}

In this section we will prove Theorem \ref{thm}. The proof will proceed by induction on the degree of monomials in $U_{\mathbb Z}(\mathcal O)$ and the following lemmas and proposition.

\begin{lem}\label{brackets}
    For all $r,s,n\in{\mathbb Z}_+$ and $j,k,l,m\in\mathbb N$
    \begin{enumerate}
        \item $\left[\left(x_j^{\Gamma,+}\right)^{(r)},\left(x_l^{\Gamma,-}\right)^{(s)}\right]$ is in the $\mathbb Z$-span of $\mathcal B$ and has degree less that $r+s$.
        
        \item $\left[\left(x_j^{\Gamma,+}\right)^{(r)},\Lambda^\Gamma_{k,m,n}\right]$ is in the $\mathbb Z$-span of $\mathcal B$ and has degree less that $r+n$.
        
        \item $\left[\Lambda^\Gamma_{k,m,n},\left(x_l^{\Gamma,-}\right)^{(s)}\right]$ is in the $\mathbb Z$-span of $\mathcal B$ and has degree less that $s+n$.
    \end{enumerate}
\end{lem}
\begin{proof}
    For (1), by Proposition \ref{straightening}(7) we have 
    \begin{equation*}
        \left[\left(x_j^{\Gamma,+}\right)^{(r)},\left(x_l^{\Gamma,-}\right)^{(s)}\right]
        =\sum_{\substack{m,n,q\in{\mathbb Z}_+\\0<m+n+q\leq\min\{r,s\}}}(-1)^{m+n+q}D^{\Gamma,-}_{m,s-m-n-q}(j,l)\Lambda^\Gamma_{j,l,n}D^{\Gamma,+}_{q,r-m-n-q}(j,l).
    \end{equation*}
    By Proposition \ref{Du1} and the definition of $D^{\Gamma,\pm}_{u,v}(j,l)$ we see that the terms of the summation are in the $\mathbb Z$-span of $\mathcal B$ and have degree $r+s-2m-n-2q<r+s$. For (2) and (3) we have, by Proposition \ref{straightening}(8) and (9),
    \begin{align*}
        \left[\left(x_j^{\Gamma,+}\right)^{(r)},\Lambda^\Gamma_{k,m,n}\right]
        &=\sum_{i=1}^n\Lambda_{k,m,n-i}^\Gamma\prod_{\substack{v_0\ldots v_u\in{\mathbb Z}_+\\\sum v_u=r\\\sum uv_u=i}}\left((u+1)D^{\Gamma,+}_{u}(j,k,m)\right)^{(v_u)},\\
        \left[\Lambda^\Gamma_{k,m,n},\left(x_l^{\Gamma,-}\right)^{(s)}\right]
        &=\sum_{i=1}^n\prod_{\substack{v_0\ldots v_u\in{\mathbb Z}_+\\\sum v_u=s\\\sum uv_u=i}}\left((u+1)D^{\Gamma,-}_{u}(l,k,m)\right)^{(v_u)}\Lambda_{k,m,n-i}^\Gamma.
    \end{align*}
    In either case the right-hand side is in the ${\mathbb Z}$-span of $\mathcal B$ and has degree $n-i+r<n+r$ or $n-i+s<n+s$ respectively by the definition of $D^{\Gamma,\pm}_{u}(j,k,m)$.
\end{proof}

The following proposition gives the $p^{\Gamma}_u(j,l)$ as ${\mathbb Z}$-linear combinations of the $h^{\Gamma}_{n}$.

\begin{prop}\label{p_u}
    Given $u\in\mathbb N$
    \begin{align*}
        p^{\Gamma}_u(j,l)
        &=\sum_{k=0}^{\left\lfloor\frac{u-1}{2}\right\rfloor}\sum_{i=0}^{u}(-1)^{k+i}\binom{u}{k}\binom{u}{i}\left(h_{(u-2i)j+(u-2k)l}^{\Gamma}\right)\nonumber\\
        &+((u+1)\mod2)\sum_{i=0}^{u}(-1)^{\frac{u}{2}+i}\binom{u-1}{\frac{u-2}{2}}\binom{u}{i}\left(h_{(u-2i)j}^{\Gamma}\right)
    \end{align*}\hfill\qedsymbol
\end{prop}

The proof is omitted due to straightforwardness.

\begin{lem}\label{LambdaLambda}
    For all $j,l\in\mathbb N$ and $k,m\in\mathbb N$
    \begin{equation*}
        \Lambda_{j,l,k}^\Gamma\Lambda_{j,l,m}^\Gamma=\binom{k+m}{k}\Lambda_{j,l,k+m}^\Gamma+u
    \end{equation*}
    where $u$ is in the ${\mathbb Z}$-span of the set $\left\{\Lambda_{i_1,n_1,q_1}^\Gamma\dots\Lambda_{i_s,n_s,q_s}^\Gamma\ \Big|\ i_1,\ldots,i_s,n_1\ldots,n_s\in\mathbb N,\ q_1,\ldots,q_s\in{\mathbb Z}_+\right\}$ and has degree less than $k+m$.
\end{lem}
\begin{proof}
    $\Lambda_{j,l,k}^\Gamma$ corresponds to $\Lambda_{k-1}$ from \cite{G} under the map $\tau:U\left((\lie g^0\otimes\mathbb C [t,t^{-1}])^\Gamma\right)\to\mathbb C[X_1,X_2,\ldots]$ given by $\tau(p_i^\Gamma(j,l))=-X_i$. Therefore, by Lemma 9.2 in \cite{G}, it suffices to show that $p_i^\Gamma(j,l)$ is an integer linear combination of $\Lambda^\Gamma_{i,n,q}$. To that end we have, for all $n\in{\mathbb Z}_+$
    \begin{align}
        p_{2n+1}^\Gamma(j,l)
        &=\sum_{i=0}^{n}\sum_{k=0}^{n}(-1)^{k+i+1}\binom{2n+1}{i}\binom{2n+1}{k}\Lambda^\Gamma_{(2n+1-2i)j,(2n+1-2k)l,1}.\label{p2n+1}\\
        p^\Gamma_{2n}(j,l)
        &=\sum_{i=0}^{n-1}\sum_{k=0}^{2n}(-1)^{k+i+1}\binom{2n-1}{i}\binom{2n}{k}\Lambda^\Gamma_{(2n-1-2i)j+(2n-2k)l,j,1}.\label{p2n}
    \end{align}
    The proofs of \eqref{p2n+1} and \eqref{p2n} are straightforward and have been omitted to shorten the text.
\end{proof}

We can now prove Theorem \ref{thm}.
\begin{proof}
    Let $B$ be the set of ordered monomials constructed from $M$. The PBW Theorem implies that $B$ is a $\mathbb C$-linearly independent set. Thus it is a ${\mathbb Z}$-linearly independent set. 
    
    The proof that the ${\mathbb Z}$-span of $B$ is $U_{\mathbb Z}(\cal O)$ will proceed by induction on the degree of monomials in $U_{\mathbb Z}(\cal O)$. Since $\Lambda^\Gamma_{j,l,1}=-\left(h^\Gamma_{j+l}-h^\Gamma_{j,l}\right)$, any degree one monomial is in the ${\mathbb Z}$-span of $B$. Now take any monomial, $m$, in $U_{\mathbb Z}(\cal O)$. If $m\in B$, then we are done. If not then either the factors of $m$ are not in the correct order or $m$ has ``repeated'' factors with the following forms
    \begin{equation}\label{repeated}
        \left(x_j^{\Gamma,\pm}\right)^{(r)}\text{ and }\left(x_j^{\Gamma,\pm}\right)^{(s)}\text{ or }\Lambda^\Gamma_{j,l,r}\text{ and }\Lambda^\Gamma_{j,l,s},\ j,l,r,s\in\mathbb N.
    \end{equation}
    If the factors of $m$ are not in the correct order, then we can rearrange the factors of $m$ using the straightening identities in Proposition \ref{straightening}. Once this is done Lemma \ref{brackets} guarantees that each rearrangement will only produce ${\mathbb Z}$-linear combinations of monomials in the correct order with lower degree. These lower degree monomials are then in the ${\mathbb Z}$-span of $B$ by the induction hypothesis.
    
    If (possibly after rearranging factors as above) $m$ contains the products of the pairs of factors in \eqref{repeated} we apply Proposition \ref{straightening}(6) or Lemma \ref{LambdaLambda} respectively to consolidate these pairs of factors into single factors with integral coefficients.
    
    In the end we see that $m\in{\mathbb Z}$-span $B$. Thus the ${\mathbb Z}$-span of $B$ is $U_{\mathbb Z}(\cal O)$ and hence $B$ is an integral basis for $U_{\mathbb Z}(\cal O)$.
\end{proof}

All that remains is to prove Proposition \ref{straightening}(7)-(9).

\subsection{Proof of Proposition \ref{straightening} (8) and (9).} We first need the following propositions.

\begin{prop}\label{pnewD}
    For all $k,u\in{\mathbb Z}_+$
    \begin{equation*}
        p_{k+u+1}^\Gamma(j,l)=\left[D^{\Gamma,+}_{u,1}(j,l),D^{\Gamma,-}_{k,1}(j,l)\right]
    \end{equation*}
\end{prop}
\begin{proof}
    To prove the claim we proceed by induction on $u$. If $u=0$ the claim is true by definition. Now assume the claim for some $u\in{\mathbb Z}_+$. Then we have
    \begin{align*}
        \left[D^{\Gamma,+}_{u+1,1}(j,l),D^{\Gamma,-}_{k,1}(j,l)\right]
        &=\frac{1}{2}\left[\left[D^{\Gamma,+}_{u,1}(j,l),\Lambda_{j,l,1}^\Gamma\right],D^{\Gamma,-}_{k,1}(j,l)\right]\\
        &=-\frac{1}{2}\left[D^{\Gamma,-}_{k,1}(j,l),\left[D^{\Gamma,+}_{u,1}(j,l),\Lambda_{j,l,1}^\Gamma\right]\right]\\
        &=\frac{1}{2}\left[D^{\Gamma,+}_{u,1}(j,l),\left[\Lambda_{j,l,1}^\Gamma,D^{\Gamma,-}_{k,1}(j,l)\right]\right]+\frac{1}{2}\left[\Lambda_{j,l,1}^\Gamma,\left[D^{\Gamma,-}_{k,1}(j,l),D^{\Gamma,+}_{u,1}(j,l)\right]\right]\\
        &\textnormal{ by the Jacobi Identitiy}\\
        &=\left[D^{\Gamma,+}_{u,1}(j,l),D^{\Gamma,-}_{k+1,1}(j,l)\right]-\frac{1}{2}\left[\Lambda_{j,l,1}^\Gamma,p_{k+u+1}^\Gamma(j,l)\right]\textnormal{ by the induction hypothesis}\\
        &=p_{k+u+1}^\Gamma(j,l)\textnormal{ by the induction hypothesis again}
    \end{align*}
\end{proof}

We now give the proof of Proposition \ref{straightening} (8). The proof of $(9)$ is similar.

\begin{proof}
We proceed by induction on $r$. The case $r=0$ is clear. In the case $r=1$ we need to show
   
    \begin{equation}\label{xj+Lambdan}
        \left(x_j^{\Gamma,+}\right)\Lambda^\Gamma_{k,m,n}=\sum_{i=0}^n\sum_{r=0}^{i}\sum_{s=0}^i(-1)^{r+s}(i+1)\binom{i}{r}\binom{i}{s}\Lambda^\Gamma_{k,m,n-i}\left(x_{j+(i-2r)k+(i-2s)m}^{\Gamma,+}\right).
    \end{equation}
    We will prove this by induction on $n$. The case $n=1$ is a straightforward calculation. Assume for some $n\in\mathbb N$ that \eqref{xj+Lambdan} holds for all $1\leq q\leq n$. Then we have
    \begin{align*}
        \left(x_j^{\Gamma,+}\right)\Lambda^\Gamma_{k,m,n}
        &=-\frac{1}{n}\sum_{i=1}^n\left(x_j^{\Gamma,+}\right)p^\Gamma_i(k,m)\Lambda^\Gamma_{k,m,n-i}\\
        &=-\frac{1}{n}\sum_{i=1}^np^\Gamma_i(k,m)\left(x_j^{\Gamma,+}\right)\Lambda^\Gamma_{k,m,n-i}\\
        &+\frac{2}{n}\sum_{i=1}^n\sum_{r=0}^i\sum_{s=0}^i(-1)^{r+s}\binom{i}{r}\binom{i}{s}\left(x_{j+(i-2r)k+(i-2s)m}^{\Gamma,+}\right)\Lambda^\Gamma_{k,m,n-i}\textnormal{ by Proposition }\ref{bracketxp}\\
        &=-\frac{1}{n}\sum_{i=1}^n\sum_{u=0}^{n-i}\sum_{r=0}^{u}\sum_{s=0}^u(-1)^{r+s}(u+1)\binom{u}{r}\binom{u}{s}p^\Gamma_i(k,m)\Lambda^\Gamma_{k,m,n-i-u}\left(x_{j+(u-2r)k+(u-2s)m}^{\Gamma,+}\right)\\
        &+\frac{2}{n}\sum_{i=1}^n\sum_{r=0}^i\sum_{s=0}^i\sum_{u=0}^{n-i}\sum_{v=0}^{u}\sum_{w=0}^u(u+1)(-1)^{r+s}(-1)^{v+w}\binom{i}{r}\binom{i}{s}\binom{u}{v}\binom{u}{w}\Lambda^\Gamma_{k,m,n-i-u}\\
        &\times\left(x_{j+(u+i-2(v+r))k+(u+i-2(w+s))m}^{\Gamma,+}\right)\textnormal{ by the induction hypothesis}\\
        &=-\frac{1}{n}\sum_{u=0}^{n-1}\sum_{r=0}^{u}\sum_{s=0}^u(-1)^{r+s}(u+1)\binom{u}{r}\binom{u}{s}\sum_{i=1}^{n-u}p^\Gamma_i(k,m)\Lambda^\Gamma_{k,m,n-i-u}\left(x_{j+(u-2r)k+(u-2s)m}^{\Gamma,+}\right)\\
        &+\frac{2}{n}\sum_{i=1}^n\sum_{r=0}^i\sum_{s=0}^i\sum_{u=i}^{n}\sum_{v=0}^{u-i}\sum_{w=0}^{u-i}(u-i+1)(-1)^{r+s}(-1)^{v+w}\binom{i}{r}\binom{i}{s}\binom{u-i}{v}\binom{u-i}{w}\Lambda^\Gamma_{k,m,n-u}\\
        &\times\left(x_{j+(u-2(v+r))k+(u-2(w+s))m}^{\Gamma,+}\right)\\
        &=\frac{1}{n}\sum_{u=0}^{n-1}\sum_{r=0}^{u}\sum_{s=0}^u(-1)^{r+s}(u+1)(n-u)\binom{u}{r}\binom{u}{s}\Lambda^\Gamma_{k,m,n-u}\left(x_{j+(u-2r)k+(u-2s)m}^{\Gamma,+}\right)\\
        &+\frac{2}{n}\sum_{i=1}^n\sum_{r=0}^i\sum_{s=0}^i\sum_{u=i}^{n}\sum_{v=r}^{u+r-i}\sum_{w=s}^{u+s-i}(u-i+1)(-1)^{r+s}(-1)^{v-r+w-s}\binom{i}{r}\binom{i}{s}\binom{u-i}{v-r}\binom{u-i}{w-s}\\
        &\times\Lambda^\Gamma_{k,m,n-u}\left(x_{j+(u-2v)k+(u-2w)m}^{\Gamma,+}\right)\\
        &=\frac{1}{n}\sum_{u=0}^{n-1}\sum_{r=0}^{u}\sum_{s=0}^u(-1)^{r+s}(u+1)(n-u)\binom{u}{r}\binom{u}{s}\Lambda^\Gamma_{k,m,n-u}\left(x_{j+(u-2r)k+(u-2s)m}^{\Gamma,+}\right)\\
        &+\frac{2}{n}\sum_{u=1}^{n}\sum_{v=0}^{u}\sum_{w=0}^{u}\sum_{i=1}^u\sum_{r=0}^{v}\sum_{s=0}^{w}(-1)^{v+w}(u-i+1)\binom{i}{r}\binom{i}{s}\binom{u-i}{v-r}\binom{u-i}{w-s}\Lambda^\Gamma_{k,m,n-u}\\
        &\times\left(x_{j+(u-2v)k+(u-2w)m}^{\Gamma,+}\right)\\
        &=\frac{1}{n}\sum_{u=0}^{n-1}\sum_{r=0}^{u}\sum_{s=0}^u(-1)^{r+s}(u+1)(n-u)\binom{u}{r}\binom{u}{s}\Lambda^\Gamma_{k,m,n-u}\left(x_{j+(u-2r)k+(u-2s)m}^{\Gamma,+}\right)\\
        &+\frac{2}{n}\sum_{u=1}^{n}\sum_{v=0}^{u}\sum_{w=0}^{u}(-1)^{v+w}\sum_{i=1}^u(u-i+1)\binom{u}{v}\binom{u}{w}\Lambda^\Gamma_{k,m,n-u}\left(x_{j+(u-2v)k+(u-2w)m}^{\Gamma,+}\right)\\
        &\textnormal{by the Chu–-Vandermonde identity}\\
        &=\Lambda^\Gamma_{k,m,n}\left(x_{j}^{\Gamma,+}\right)\\
        &+\frac{1}{n}\sum_{u=1}^{n}\sum_{r=0}^{u}\sum_{s=0}^u(-1)^{r+s}(u+1)(n-u)\binom{u}{r}\binom{u}{s}\Lambda^\Gamma_{k,m,n-u}\left(x_{j+(u-2r)k+(u-2s)m}^{\Gamma,+}\right)\\
        &+\frac{2}{n}\sum_{u=1}^{n}\sum_{v=0}^{u}\sum_{w=0}^{u}(-1)^{v+w}\left(u(u+1)-\frac{u(u+1)}{2}\right)\binom{u}{v}\binom{u}{w}\Lambda^\Gamma_{k,m,n-u}\left(x_{j+(u-2v)k+(u-2w)m}^{\Gamma,+}\right)\\
        &=\Lambda^\Gamma_{k,m,n}\left(x_{j}^{\Gamma,+}\right)\\
        &+\frac{1}{n}\sum_{u=1}^{n}\sum_{r=0}^{u}\sum_{s=0}^u(-1)^{r+s}(u+1)(n-u)\binom{u}{r}\binom{u}{s}\Lambda^\Gamma_{k,m,n-u}\left(x_{j+(u-2r)k+(u-2s)m}^{\Gamma,+}\right)\\
        &+\frac{1}{n}\sum_{u=1}^{n}\sum_{v=0}^{u}\sum_{w=0}^{u}(-1)^{v+w}u(u+1)\binom{u}{v}\binom{u}{w}\Lambda^\Gamma_{k,m,n-u}\left(x_{j+(u-2v)k+(u-2w)m}^{\Gamma,+}\right)\\
        &=\Lambda^\Gamma_{k,m,n}\left(x_{j}^{\Gamma,+}\right)+\sum_{u=1}^{n}\sum_{r=0}^{u}\sum_{s=0}^u(-1)^{r+s}(u+1)\binom{u}{r}\binom{u}{s}\Lambda^\Gamma_{k,m,n-u}\left(x_{j+(u-2r)k+(u-2s)m}^{\Gamma,+}\right)\\
        &=\sum_{u=0}^{n}\sum_{r=0}^{u}\sum_{s=0}^u(-1)^{r+s}(u+1)\binom{u}{r}\binom{u}{s}\Lambda^\Gamma_{k,m,n-u}\left(x_{j+(u-2r)k+(u-2s)m}^{\Gamma,+}\right)
    \end{align*}
    This concludes the proof of the $r=1$ case. Assume the proposition for some $r\in\mathbb N$. Then
    \begin{align*}
        (r+1)\left(x_{j}^{\Gamma,+}\right)^{(r+1)}\Lambda_{k,m,n}^\Gamma
        &=\left(x_{j}^{\Gamma,+}\right)\left(x_{j}^{\Gamma,+}\right)^{(r)}\Lambda_{k,m,n}^\Gamma\\
        &=\sum_{i=0}^n\left(x_{j}^{\Gamma,+}\right)\Lambda_{k,m,n-i}^\Gamma\prod_{\substack{v_0\ldots v_u\in{\mathbb Z}_+\\\sum v_u=r\\\sum uv_u=i}}\left((u+1)D^{\Gamma,+}_{u}(j,k,m)\right)^{(v_u)}\\
        &\textnormal{by the induction hypothesis}\\
        &=\sum_{i=0}^n\sum_{s=0}^{n-i}\Lambda^\Gamma_{k,m,n-i-s}(s+1)D_s^{\Gamma,+}(j,k,m)\prod_{\substack{v_0\ldots v_u\in{\mathbb Z}_+\\\sum v_u=r\\\sum uv_u=i}}\left((u+1)D^{\Gamma,+}_{u}(j,k,m)\right)^{(v_u)}\\
        &=\sum_{i=0}^n\sum_{s=0}^{n-i}\Lambda^\Gamma_{k,m,n-i-s}\prod_{\substack{v_0\ldots v_u\in{\mathbb Z}_+\\u\neq s\\\sum v_u=r-v_s\\\sum uv_u=i-sv_s}}\left((u+1)D^{\Gamma,+}_{u}(j,k,m)\right)^{(v_u)}\\
        &\times\left((s+1)D^{\Gamma,+}_{s}(j,k,m)\right)^{(v_s)}(s+1)D_s^{\Gamma,+}(j,k,m)\\
        &=\sum_{i=0}^n\sum_{s=0}^{n-i}\Lambda^\Gamma_{k,m,n-i-s}\prod_{\substack{v_0\ldots v_u\in{\mathbb Z}_+\\u\neq s\\\sum v_u=r-v_s\\\sum uv_u=i-sv_s}}\left((u+1)D^{\Gamma,+}_{u}(j,k,m)\right)^{(v_u)}\\
        &\times(v_s+1)\left((s+1)D^{\Gamma,+}_{s}(j,k,m)\right)^{(v_s+1)}\\
        &=\sum_{i=0}^n\sum_{s=i}^{n}\Lambda^\Gamma_{k,m,n-s}\prod_{\substack{v_0\ldots v_u\in{\mathbb Z}_+\\\sum v_u=r+1\\\sum uv_u=s}}(v_{s-i})\left((u+1)D^{\Gamma,+}_{u}(j,k,m)\right)^{(v_u)}\\
        &=\sum_{s=0}^{n}\Lambda^\Gamma_{k,m,n-s}\prod_{\substack{v_0\ldots v_u\in{\mathbb Z}_+\\\sum v_u=r+1\\\sum uv_u=s}}\sum_{i=0}^s(v_{s-i})\left((u+1)D^{\Gamma,+}_{u}(j,k,m)\right)^{(v_u)}\\
        &=\sum_{s=0}^{n}\Lambda^\Gamma_{k,m,n-s}\prod_{\substack{v_0\ldots v_u\in{\mathbb Z}_+\\\sum v_u=r+1\\\sum uv_u=s}}\sum_{i=0}^s(v_{s})\left((u+1)D^{\Gamma,+}_{u}(j,k,m)\right)^{(v_u)}\\
        &=(r+1)\sum_{s=0}^{n}\Lambda^\Gamma_{k,m,n-s}\prod_{\substack{v_0\ldots v_u\in{\mathbb Z}_+\\\sum v_u=r+1\\\sum uv_u=s}}\left((u+1)D^{\Gamma,+}_{u}(j,k,m)\right)^{(v_u)}\\
    \end{align*}
\end{proof}

\subsection{Proof of Proposition \ref{straightening} (7)}
We will proceed following the method of proof of Lemma 5.4 in \cite{C}, which was in turn modeled on the proof of Lemma 5.1 in \cite{CP}. We begin with some necessary propositions.

\begin{prop}\label{bracketxp}
    For all $i,j,k,l,m\in\mathbb N$
    \begin{enumerate}
        \item 
        \begin{equation*}
            \left[\left(x_j^{\Gamma,+}\right),p_i^\Gamma(k,m)\right]=-2D_i^{\Gamma,+}(j,k,m)
        \end{equation*}
        \item 
        \begin{equation*}
            \left[p_i^\Gamma(k,m),\left(x_l^{\Gamma,-}\right)\right]=-2D_i^{\Gamma,-}(l,k,m)
        \end{equation*}
    \end{enumerate}
\end{prop}
\begin{proof}
    We will only prove $(1)$ in detail, because the proof of $(2)$ is similar. By Proposition \ref{p_u} we have
    \begin{align*}
        \left[\left(x_j^{\Gamma,+}\right),p_i^\Gamma(k,m)\right]
        &=\sum_{r=0}^{\left\lfloor\frac{i-1}{2}\right\rfloor}\sum_{s=0}^i(-1)^{r+s}\binom{i}{r}\binom{i}{s}\left[\left(x_j^{\Gamma,+}\right),\left(h^\Gamma_{(i-2s)k+(i-2r)m}\right)\right]\\
        &+((i+1)\mod2)\sum_{s=0}^i(-1)^{\frac{i}{2}+s}\binom{i-1}{\frac{i-2}{2}}\binom{i}{s}\left[\left(x_j^{\Gamma,+}\right),\left(h^\Gamma_{(i-2s)k}\right)\right]\\
        &=-2\sum_{r=0}^{\left\lfloor\frac{i-1}{2}\right\rfloor}\sum_{s=0}^i(-1)^{r+s}\binom{i}{r}\binom{i}{s}\left(\left(x_{j+(i-2s)k+(i-2r)m}^{\Gamma,+}\right)+\left(x_{j-(i-2s)k-(i-2r)m}^{\Gamma,+}\right)\right)\\
        &-2((i+1)\mod2)\sum_{s=0}^i(-1)^{\frac{i}{2}+s}\binom{i-1}{\frac{i-2}{2}}\binom{i}{s}\left(\left(x_{j+(i-2s)k}^{\Gamma,+}\right)+\left(x_{j-(i-2s)k}^{\Gamma,+}\right)\right)\\
        &=-2\sum_{r=0}^{\left\lfloor\frac{i-1}{2}\right\rfloor}\sum_{s=0}^i(-1)^{r+s}\binom{i}{r}\binom{i}{s}\left(x_{j+(i-2s)k+(i-2r)m}^{\Gamma,+}\right)\\
        &-2\sum_{r=0}^{\left\lfloor\frac{i-1}{2}\right\rfloor}\sum_{s=0}^i(-1)^{r+s}\binom{i}{r}\binom{i}{s}\left(x_{j-(i-2s)k-(i-2r)m}^{\Gamma,+}\right)\\
        &-2((i+1)\mod2)\sum_{s=0}^i(-1)^{\frac{i}{2}+s}\binom{i-1}{\frac{i-2}{2}}\binom{i}{s}\left(x_{j+(i-2s)k}^{\Gamma,+}\right)\\
        &-2((i+1)\mod2)\sum_{s=0}^i(-1)^{\frac{i}{2}+s}\binom{i-1}{\frac{i-2}{2}}\binom{i}{s}\left(x_{j-(i-2s)k}^{\Gamma,+}\right)\\
        &=-2\sum_{r=0}^{\left\lfloor\frac{i-1}{2}\right\rfloor}\sum_{s=0}^i(-1)^{r+s}\binom{i}{r}\binom{i}{s}\left(x_{j+(i-2s)k+(i-2r)m}^{\Gamma,+}\right)\\
        &-2\sum_{r=\left\lfloor\frac{i+2}{2}\right\rfloor}^{i}\sum_{s=0}^i(-1)^{i-r+i-s}\binom{i}{r}\binom{i}{s}\left(x_{j+(i-2s)k+(i-2r)m}^{\Gamma,+}\right)\\
        &-2((i+1)\mod2)\sum_{s=0}^i(-1)^{\frac{i}{2}+s}\binom{i-1}{\frac{i-2}{2}}\binom{i}{s}\left(x_{j+(i-2s)k}^{\Gamma,+}\right)\\
        &-2((i+1)\mod2)\sum_{s=0}^i(-1)^{\frac{3i}{2}-s}\binom{i-1}{\frac{i-2}{2}}\binom{i}{s}\left(x_{j+(i-2s)k}^{\Gamma,+}\right)\\
        &=-2\sum_{r=0}^{\left\lfloor\frac{i-1}{2}\right\rfloor}\sum_{s=0}^i(-1)^{r+s}\binom{i}{r}\binom{i}{s}\left(x_{j+(i-2s)k+(i-2r)m}^{\Gamma,+}\right)\\
        &-2\sum_{r=\left\lfloor\frac{i+2}{2}\right\rfloor}^{i}\sum_{s=0}^i(-1)^{r+s}\binom{i}{r}\binom{i}{s}\left(x_{j+(i-2s)k+(i-2r)m}^{\Gamma,+}\right)\\
        &-2((i+1)\mod2)\sum_{s=0}^i(-1)^{\frac{i}{2}+s}\binom{i}{\frac{i}{2}}\binom{i}{s}\left(x_{j+(i-2s)k}^{\Gamma,+}\right)\\
        &=-2\sum_{r=0}^{i}\sum_{s=0}^i(-1)^{r+s}\binom{i}{r}\binom{i}{s}\left(x_{j+(i-2s)k+(i-2r)m}^{\Gamma,+}\right)
    \end{align*}
\end{proof}

\begin{prop}\label{bracketpD}
    For all $i,j,k,m\in\mathbb N$ and $u\in{\mathbb Z}_+$ 
    \begin{equation*}
        \left[p_m^\Gamma(j,l),D_{u,1}^{\Gamma,+}(j,l)\right]=2D_{m+u,1}^{\Gamma,+}(j,l)
    \end{equation*}
\end{prop}
\begin{proof}
    We will proceed by induction on $u$. The case $u=0$ is Proposition \ref{bracketxp}. Assume the claim for some $u\in\mathbb N$. Then we have
    \begin{align*}
        \left[p_m^\Gamma(j,l),D_{u+1,1}^{\Gamma,+}(j,l)\right]
        &=\frac{1}{2}\left[p_m^\Gamma(j,l),\left[D_{u,1}^{\Gamma,+}(j,l),\Lambda^\Gamma_{j,l,1}\right]\right]\\
        &=-\frac{1}{2}\left(\left[D_{u,1}^{\Gamma,+}(j,l),\left[\Lambda^\Gamma_{j,l,1},p_m^\Gamma(j,l)\right]\right]+\left[\Lambda^\Gamma_{j,l,1},\left[p_m^\Gamma(j,l),D_{u,1}^{\Gamma,+}(j,l)\right]\right]\right)\\
        &\textnormal{ by the Jacobi Identity}\\
        &=-\frac{1}{2}\left[\Lambda^\Gamma_{j,l,1},\left[p_m^\Gamma(j,l),D_{u,1}^{\Gamma,+}(j,l)\right]\right]\\
        &=-\left[\Lambda^\Gamma_{j,l,1},D_{m+u,1}^{\Gamma,+}(j,l)\right]\textnormal{ by the induction hypothesis}\\
        &=\left[D_{m+u,1}^{\Gamma,+}(j,l),\Lambda^\Gamma_{j,l,1}\right]\\
        &=2D_{m+u+1,1}^{\Gamma,+}(j,l)
    \end{align*}
\end{proof}

\begin{prop}\label{Du1Lambda}
    For all $u,n\in{\mathbb Z}_+$ and $j,l\in\mathbb N$
    \begin{equation*}
        D^{\Gamma,+}_{u,1}(j,l)\Lambda^\Gamma_{j,l,n}=\sum_{i=0}^n(i+1)\Lambda^\Gamma_{j,l,n-i}D^{\Gamma,+}_{i+u,1}(j,l)
    \end{equation*}
\end{prop}
\begin{proof}
    We will proceed by induction on $u$. The case $u=0$ is Proposition \ref{straightening}(8). Assume the claim for some $u\in{\mathbb Z}_+$. Then we have
    \begin{align*}
        \left[D^{\Gamma,+}_{u+1,1}(j,l),\Lambda^\Gamma_{j,l,n}\right]
        &=\frac{1}{2}\left[\left[D^{\Gamma,+}_{u,1}(j,l),\Lambda^\Gamma_{j,l,1}\right],\Lambda^\Gamma_{j,l,n}\right]\\
        &=\frac{1}{2}\left[D^{\Gamma,+}_{u,1}(j,l),\left[\Lambda^\Gamma_{j,l,1},\Lambda^\Gamma_{j,l,n}\right]\right]+\frac{1}{2}\left[\Lambda^\Gamma_{j,l,1},\left[\Lambda^\Gamma_{j,l,n},D^{\Gamma,+}_{u,1}(j,l)\right]\right]\\
        &\textnormal{ by the Jacobi Identity}\\
        &=-\frac{1}{2}\sum_{i=1}^n(i+1)\left[\Lambda^\Gamma_{j,l,1},\Lambda^\Gamma_{j,l,n-i}D^{\Gamma,+}_{i+u,1}(j,l)\right]\textnormal{ by the induction hypothesis}\\
        &=-\frac{1}{2}\sum_{i=1}^n(i+1)\Lambda^\Gamma_{j,l,1}\Lambda^\Gamma_{j,l,n-i}D^{\Gamma,+}_{i+u,1}(j,l)+\frac{1}{2}\sum_{i=1}^n(i+1)\Lambda^\Gamma_{j,l,n-i}D^{\Gamma,+}_{i+u,1}(j,l)\Lambda^\Gamma_{j,l,1}\\
        &=\frac{1}{2}\sum_{i=1}^n(i+1)\Lambda^\Gamma_{j,l,n-i}\left[D^{\Gamma,+}_{i+u,1}(j,l),\Lambda^\Gamma_{j,l,1}\right]\\
        &=\sum_{i=1}^n(i+1)\Lambda^\Gamma_{j,l,n-i}D^{\Gamma,+}_{i+u+1,1}(j,l)
    \end{align*}
\end{proof}

\begin{prop}\label{LambdaD+}
    For all $i,k\in{\mathbb Z}_+$ and $j,l\in\mathbb N$ we have
    \begin{equation*}
        \Lambda^\Gamma_{j,l,i}D_{k,1}^{\Gamma,+}(j,l)=D_{k,1}^{\Gamma,+}(j,l)\Lambda^\Gamma_{j,l,i}-2D_{k+1,1}^{\Gamma,+}(j,l)\Lambda^\Gamma_{j,l,i-1}+D_{k+2,1}^{\Gamma,+}(j,l)\Lambda^\Gamma_{j,l,i-2}
    \end{equation*}
\end{prop}
\begin{proof}
    We will proceed by induction on $i$. The cases $i=0,1$ are straightforward. Assume the claim for some $i\in\mathbb N$. Then we have
    \begin{align*}
        (i+1)\Lambda^\Gamma_{j,l,i+1}D_{k,1}^{\Gamma,+}(j,l)
        &=-\sum_{m=1}^{i+1}p_m^\Gamma(j,l)\Lambda^\Gamma_{j,l,i+1-m}D_{k,1}^{\Gamma,+}(j,l)\\
        &=-\sum_{m=1}^{i+1}p_m^\Gamma(j,l)\left(D_{k,1}^{\Gamma,+}(j,l)\Lambda^\Gamma_{j,l,i+1-m}-2D_{k+1,1}^{\Gamma,+}(j,l)\Lambda^\Gamma_{j,l,i-m}+D_{k+2,1}^{\Gamma,+}(j,l)\Lambda^\Gamma_{j,l,i-1-m}\right)\\
        &\textnormal{ by the induction hypothesis}\\
        &=-\sum_{m=1}^{i+1}p_m^\Gamma(j,l)D_{k,1}^{\Gamma,+}(j,l)\Lambda^\Gamma_{j,l,i+1-m}+2\sum_{m=1}^{i+1}p_m^\Gamma(j,l)D_{k+1,1}^{\Gamma,+}(j,l)\Lambda^\Gamma_{j,l,i-m}\\
        &-\sum_{m=1}^{i+1}p_m^\Gamma(j,l)D_{k+2,1}^{\Gamma,+}(j,l)\Lambda^\Gamma_{j,l,i-1-m}\\
        &=-D_{k,1}^{\Gamma,+}(j,l)\sum_{m=1}^{i+1}p_m^\Gamma(j,l)\Lambda^\Gamma_{j,l,i+1-m}-2\sum_{m=1}^{i+1}D_{k+m,1}^{\Gamma,+}(j,l)\Lambda^\Gamma_{j,l,i+1-m}\\
        &+2D_{k+1,1}^{\Gamma,+}(j,l)\sum_{m=1}^{i+1}p_m^\Gamma(j,l)\Lambda^\Gamma_{j,l,i-m}+4\sum_{m=1}^{i+1}D_{k+m+1,1}^{\Gamma,+}(j,l)\Lambda^\Gamma_{j,l,i-m}\\
        &-D_{k+2,1}^{\Gamma,+}(j,l)\sum_{m=1}^{i+1}p_m^\Gamma(j,l)\Lambda^\Gamma_{j,l,i-1-m}-2\sum_{m=1}^{i+1}D_{k+m+2,1}^{\Gamma,+}(j,l)\Lambda^\Gamma_{j,l,i-1-m}\\
        &\textnormal{ by Proposition }\ref{bracketpD}\\
        &=(i+1)D_{k,1}^{\Gamma,+}(j,l)\Lambda^\Gamma_{j,l,i+1}-2D_{k+1,1}^{\Gamma,+}(j,l)\Lambda^\Gamma_{j,l,i}-2D_{k+2,1}^{\Gamma,+}(j,l)\Lambda^\Gamma_{j,l,i-1}\\
        &-2\sum_{m=3}^{i+1}D_{k+m,1}^{\Gamma,+}(j,l)\Lambda^\Gamma_{j,l,i+1-m}-2iD_{k+1,1}^{\Gamma,+}(j,l)\Lambda^\Gamma_{j,l,i}+4\sum_{m=2}^{i+2}D_{k+m,1}^{\Gamma,+}(j,l)\Lambda^\Gamma_{j,l,i+1-m}\\
        &+(i-1)D_{k+2,1}^{\Gamma,+}(j,l)\Lambda^\Gamma_{j,l,i-1}-2\sum_{m=3}^{i+3}D_{k+m,1}^{\Gamma,+}(j,l)\Lambda^\Gamma_{j,l,i+1-m}\\
        &=(i+1)D_{k,1}^{\Gamma,+}(j,l)\Lambda^\Gamma_{j,l,i+1}-2\sum_{m=3}^{i+1}D_{k+m,1}^{\Gamma,+}(j,l)\Lambda^\Gamma_{j,l,i+1-m}-2(i+1)D_{k+1,1}^{\Gamma,+}(j,l)\Lambda^\Gamma_{j,l,i}\\
        &+4D_{k+2,1}^{\Gamma,+}(j,l)\Lambda^\Gamma_{j,l,i-1}+4\sum_{m=3}^{i+1}D_{k+m,1}^{\Gamma,+}(j,l)\Lambda^\Gamma_{j,l,i+1-m}+(i-3)D_{k+2,1}^{\Gamma,+}(j,l)\Lambda^\Gamma_{j,l,i-1}\\
        &-2\sum_{m=3}^{i+1}D_{k+m,1}^{\Gamma,+}(j,l)\Lambda^\Gamma_{j,l,i+1-m}\\
        &=(i+1)D_{k,1}^{\Gamma,+}(j,l)\Lambda^\Gamma_{j,l,i+1}-2(i+1)D_{k+1,1}^{\Gamma,+}(j,l)\Lambda^\Gamma_{j,l,i}+(i+1)D_{k+2,1}^{\Gamma,+}(j,l)\Lambda^\Gamma_{j,l,i-1}
    \end{align*}
\end{proof}

\begin{prop}\label{uD}
    For all $j,l\in\mathbb N$ and $u,v\in{\mathbb Z}_+$
    \begin{enumerate}
        \item
        \begin{equation*}
            uD_{u,v}^{\Gamma,\pm}(j,l)=\sum_{i=0}^uiD_{i,1}^{\Gamma,\pm}(j,l)D_{u-i,v-1}^{\Gamma,\pm}(j,l)
        \end{equation*}
        \item
        \begin{equation*}
            (u+v)D_{u,v}^{\Gamma,\pm}(j,l)=\sum_{i=0}^u(i+1)D_{i,1}^{\Gamma,\pm}(j,l)D_{u-i,v-1}^{\Gamma,\pm}(j,l)
        \end{equation*}
    \end{enumerate}
\end{prop}
\begin{proof}
    We will proceed by induction on $v$. The cases $v=0,1$ are straightforward. Assume the formula for some $v\in{\mathbb Z}_+$. Then, for $(1)$, we have
    \begin{align*}
        uD_{u,v+1}^{\Gamma,\pm}(j,l)
        &=\frac{u}{v+1}\sum_{i=0}^uD_{i,1}^{\Gamma,\pm}(j,l)D_{u-i,v}^{\Gamma,\pm}(j,l)\\
        &=\frac{1}{v+1}\sum_{i=0}^uD_{i,1}^{\Gamma,\pm}(j,l)(u-i)D_{u-i,v}^{\Gamma,\pm}(j,l)+\frac{1}{v+1}\sum_{i=0}^uiD_{i,1}^{\Gamma,\pm}(j,l)D_{u-i,v}^{\Gamma,\pm}(j,l)\\
        &=\frac{1}{v+1}\sum_{i=0}^uD_{i,1}^{\Gamma,\pm}(j,l)\sum_{k=0}^{u-i}kD_{k,1}^{\Gamma,\pm}(j,l)D_{u-i-k,v-1}^{\Gamma,\pm}(j,l)+\frac{1}{v+1}\sum_{i=0}^uiD_{i,1}^{\Gamma,\pm}(j,l)D_{u-i,v}^{\Gamma,\pm}(j,l)\\
        &\textnormal{ by the induction hypothesis}\\
        &=\frac{1}{v+1}\sum_{k=0}^{u}kD_{k,1}^{\Gamma,\pm}(j,l)\sum_{i=0}^{u-k}D_{i,1}^{\Gamma,\pm}(j,l)D_{u-i-k,v-1}^{\Gamma,\pm}(j,l)+\frac{1}{v+1}\sum_{i=0}^uiD_{i,1}^{\Gamma,\pm}(j,l)D_{u-i,v}^{\Gamma,\pm}(j,l)\\
        &=\frac{v}{v+1}\sum_{k=0}^{u}kD_{k,1}^{\Gamma,\pm}(j,l)D_{u-k,v}^{\Gamma,\pm}(j,l)+\frac{1}{v+1}\sum_{i=0}^uiD_{i,1}^{\Gamma,\pm}(j,l)D_{u-i,v}^{\Gamma,\pm}(j,l)\\
        &=\sum_{i=0}^uiD_{i,1}^{\Gamma,\pm}(j,l)D_{u-i,v}^{\Gamma,\pm}(j,l).
    \end{align*}
    For $(2)$, we have
    \begin{align*}
        (u+v)D_{u,v}^{\Gamma,\pm}(j,l)
        &=uD_{u,v}^{\Gamma,\pm}(j,l)+vD_{u,v}^{\Gamma,\pm}(j,l)\\
        &=\sum_{i=0}^uiD_{i,1}^{\Gamma,\pm}(j,l)D_{u-i,v-1}^{\Gamma,\pm}(j,l)+\sum_{i=0}^uD_{i,1}^{\Gamma,\pm}(j,l)D_{u-i,v-1}^{\Gamma,\pm}(j,l)\textnormal{ by }(1)\\
        &=\sum_{i=0}^u(i+1)D_{i,1}^{\Gamma,\pm}(j,l)D_{u-i,v-1}^{\Gamma,\pm}(j,l)
    \end{align*}
\end{proof}

\begin{lem}\label{LambdaD+x-}
    For all $n,v\in{\mathbb Z}_+$ and $j,l\in\mathbb N$
    \begin{align*}
        \sum_{i=0}^n\Lambda^\Gamma_{j,l,i}D^{\Gamma,+}_{n-i,v}(j,l)\left(x^{\Gamma,-}_l\right)
        &=-(n+1)\sum_{u=0}^{n+1}\Lambda^\Gamma_{j,l,n+1-u}D^{\Gamma,+}_{u,v-1}(j,l)\\
        &+\sum_{m=0}^{n}\sum_{k=0}^{n-m}(m+1)D^{\Gamma,-}_{m,1}(j,l)\Lambda^\Gamma_{j,l,n-m-k}D^{\Gamma,+}_{k,v}(j,l)
    \end{align*}
\end{lem}

\begin{proof}
    We will proceed by induction on $v$. If $v=0$ we have
    \begin{equation*}
        \Lambda^\Gamma_{j,l,n}\left(x^{\Gamma,-}_l\right)=\sum_{m=0}^n(m+1)D^{\Gamma,-}_{m}(l,j,l)\Lambda^\Gamma_{j,l,n-m}.
    \end{equation*}
    Assume the result for some $v\in{\mathbb Z}_+$. Then we have
    \begin{align*}
        &(v+1)\sum_{i=0}^n\Lambda^\Gamma_{j,l,i}D^{\Gamma,+}_{n-i,v+1}(j,l)\left(x^{\Gamma,-}_l\right)
       {\Gamma,+}_{n-i-k,v}(j,l)\left(x^{\Gamma,-}_l\right)\\
        &=\sum_{i=0}^n\sum_{k=0}^{n-i}\Lambda^\Gamma_{j,l,i}D^{\Gamma,+}_{k,1}(j,l)D^{\Gamma,+}_{n-i-k,v}(j,l)\left(x^{\Gamma,-}_l\right)\\
        &=\sum_{k=0}^n\sum_{i=0}^{n-k}\Lambda^\Gamma_{j,l,i}D^{\Gamma,+}_{k,1}(j,l)D^{\Gamma,+}_{n-i-k,v}(j,l)\left(x^{\Gamma,-}_l\right)\\
        &=\sum_{k=0}^n\sum_{i=0}^{n-k}\Big(D^{\Gamma,+}_{k,1}(j,l)\Lambda^\Gamma_{j,l,i}-2D^{\Gamma,+}_{k+1,1}(j,l)\Lambda^\Gamma_{j,l,i-1}\\
        &+D^{\Gamma,+}_{k+2,1}(j,l)\Lambda^\Gamma_{j,l,i-2}\Big)D^{\Gamma,+}_{n-i-k,v}(j,l)\left(x^{\Gamma,-}_l\right)\\
        &\textnormal{ by Proposition }\ref{LambdaD+}\\
        &=\sum_{k=0}^{n}D^{\Gamma,+}_{k,1}(j,l)\sum_{i=0}^{n-k}\Lambda^\Gamma_{j,l,i}D^{\Gamma,+}_{n-i-k,v}(j,l)\left(x^{\Gamma,-}_l\right)\\
        &-2\sum_{k=0}^nD^{\Gamma,+}_{k+1,1}(j,l)\sum_{i=0}^{n-k}\Lambda^\Gamma_{j,l,i-1}D^{\Gamma,+}_{n-i-k,v}(j,l)\left(x^{\Gamma,-}_l\right)\\
        &+\sum_{k=0}^nD^{\Gamma,+}_{k+2,1}(j,l)\sum_{i=0}^{n-k}\Lambda^\Gamma_{j,l,i-2}D^{\Gamma,+}_{n-i-k,v}(j,l)\left(x^{\Gamma,-}_l\right)\\
        &=\sum_{k=0}^{n}D^{\Gamma,+}_{k,1}(j,l)\sum_{i=0}^{n-k}\Lambda^\Gamma_{j,l,i}D^{\Gamma,+}_{n-i-k,v}(j,l)\left(x^{\Gamma,-}_l\right)\\
        &-2\sum_{k=0}^nD^{\Gamma,+}_{k+1,1}(j,l)\sum_{i=0}^{n-k-1}\Lambda^\Gamma_{j,l,i}D^{\Gamma,+}_{n-i-k-1,v}(j,l)\left(x^{\Gamma,-}_l\right)\\
        &+\sum_{k=0}^nD^{\Gamma,+}_{k+2,1}(j,l)\sum_{i=0}^{n-k-2}\Lambda^\Gamma_{j,l,i}D^{\Gamma,+}_{n-i-k-2,v}(j,l)\left(x^{\Gamma,-}_l\right)\\
        &=\sum_{k=0}^{n}D^{\Gamma,+}_{k,1}(j,l)\Bigg(-(n-k+1)\sum_{u=0}^{n-k+1}\Lambda^\Gamma_{j,l,n-k+1-u}D^{\Gamma,+}_{u,v-1}(j,l)\\
        &+\sum_{m=0}^{n-k}\sum_{r=0}^{n-k-m}(m+1)D^{\Gamma,-}_{m,1}(j,l)\Lambda^\Gamma_{j,l,n-k-m-r}D^{\Gamma,+}_{r,v}(j,l)\Bigg)\\
        &-2\sum_{k=0}^nD^{\Gamma,+}_{k+1,1}(j,l)\Bigg(-(n-k)\sum_{u=0}^{n-k}\Lambda^\Gamma_{j,l,n-k-u}D^{\Gamma,+}_{u,v-1}(j,l)\\
        &+\sum_{m=0}^{n-k-1}\sum_{r=0}^{n-k-1-m}(m+1)D^{\Gamma,-}_{m,1}(j,l)\Lambda^\Gamma_{j,l,n-k-1-m-r}D^{\Gamma,+}_{r,v}(j,l)\Bigg)\\
        &+\sum_{k=0}^nD^{\Gamma,+}_{k+2,1}(j,l)\Bigg(-(n-k-1)\sum_{u=0}^{n-k-1}\Lambda^\Gamma_{j,l,n-k-1-u}D^{\Gamma,+}_{u,v-1}(j,l)\\
        &+\sum_{m=0}^{n-k-2}\sum_{r=0}^{n-k-2-m}(m+1)D^{\Gamma,-}_{m,1}(j,l)\Lambda^\Gamma_{j,l,n-k-2-m-r}D^{\Gamma,+}_{r,v}(j,l)\Bigg)\\
        &\textnormal{ by the induction hypothesis}\\
        &=-\sum_{k=0}^{n}D^{\Gamma,+}_{k,1}(j,l)(n-k+1)\sum_{u=0}^{n-k+1}\Lambda^\Gamma_{j,l,n-k+1-u}D^{\Gamma,+}_{u,v-1}(j,l)\\
        &+\sum_{k=0}^{n}D^{\Gamma,+}_{k,1}(j,l)\sum_{m=0}^{n-k}\sum_{r=0}^{n-k-m}(m+1)D^{\Gamma,-}_{m,1}(j,l)\Lambda^\Gamma_{j,l,n-k-m-r}D^{\Gamma,+}_{r,v}(j,l)\\
        &+2\sum_{k=0}^nD^{\Gamma,+}_{k+1,1}(j,l)(n-k)\sum_{u=0}^{n-k}\Lambda^\Gamma_{j,l,n-k-u}D^{\Gamma,+}_{u,v-1}(j,l)\\
        &-2\sum_{k=0}^nD^{\Gamma,+}_{k+1,1}(j,l)\sum_{m=0}^{n-k-1}\sum_{r=0}^{n-k-1-m}(m+1)D^{\Gamma,-}_{m,1}(j,l)\Lambda^\Gamma_{j,l,n-k-1-m-r}D^{\Gamma,+}_{r,v}(j,l)\\
        &-\sum_{k=0}^nD^{\Gamma,+}_{k+2,1}(j,l)(n-k-1)\sum_{u=0}^{n-k-1}\Lambda^\Gamma_{j,l,n-k-1-u}D^{\Gamma,+}_{u,v-1}(j,l)\\
        &+\sum_{k=0}^nD^{\Gamma,+}_{k+2,1}(j,l)\sum_{m=0}^{n-k-2}\sum_{r=0}^{n-k-2-m}(m+1)D^{\Gamma,-}_{m,1}(j,l)\Lambda^\Gamma_{j,l,n-k-2-m-r}D^{\Gamma,+}_{r,v}(j,l)\\
        &=-\sum_{k=0}^{n}D^{\Gamma,+}_{k,1}(j,l)(n-k+1)\sum_{u=0}^{n-k+1}\Lambda^\Gamma_{j,l,n-k+1-u}D^{\Gamma,+}_{u,v-1}(j,l)\\
        &+\sum_{k=0}^{n}D^{\Gamma,+}_{k,1}(j,l)\sum_{m=0}^{n-k}\sum_{r=0}^{n-k-m}(m+1)D^{\Gamma,-}_{m,1}(j,l)\Lambda^\Gamma_{j,l,n-k-m-r}D^{\Gamma,+}_{r,v}(j,l)\\
        &+2\sum_{k=1}^{n+1}D^{\Gamma,+}_{k,1}(j,l)(n-k+1)\sum_{u=0}^{n-k+1}\Lambda^\Gamma_{j,l,n-k-u+1}D^{\Gamma,+}_{u,v-1}(j,l)\\
        &-2\sum_{k=1}^{n+1}D^{\Gamma,+}_{k,1}(j,l)\sum_{m=0}^{n-k}\sum_{r=0}^{n-k-m}(m+1)D^{\Gamma,-}_{m,1}(j,l)\Lambda^\Gamma_{j,l,n-k-m-r}D^{\Gamma,+}_{r,v}(j,l)\\
        &-\sum_{k=2}^{n+2}D^{\Gamma,+}_{k,1}(j,l)(n-k+1)\sum_{u=0}^{n-k+1}\Lambda^\Gamma_{j,l,n-k+1-u}D^{\Gamma,+}_{u,v-1}(j,l)\\
        &+\sum_{k=2}^{n+2}D^{\Gamma,+}_{k,1}(j,l)\sum_{m=0}^{n-k}\sum_{r=0}^{n-k-m}(m+1)D^{\Gamma,-}_{m,1}(j,l)\Lambda^\Gamma_{j,l,n-k-m-r}D^{\Gamma,+}_{r,v}(j,l)\\
        &=-D^{\Gamma,+}_{0,1}(j,l)(n+1)\sum_{u=0}^{n+1}\Lambda^\Gamma_{j,l,n+1-u}D^{\Gamma,+}_{u,v-1}(j,l)-D^{\Gamma,+}_{1,1}(j,l)n\sum_{u=0}^{n}\Lambda^\Gamma_{j,l,n-u}D^{\Gamma,+}_{u,v-1}(j,l)\\
        &-\sum_{k=2}^{n}D^{\Gamma,+}_{k,1}(j,l)(n-k+1)\sum_{u=0}^{n-k+1}\Lambda^\Gamma_{j,l,n-k+1-u}D^{\Gamma,+}_{u,v-1}(j,l)\\
        &+D^{\Gamma,+}_{0,1}(j,l)\sum_{m=0}^{n}\sum_{r=0}^{n-m}(m+1)D^{\Gamma,-}_{m,1}(j,l)\Lambda^\Gamma_{j,l,n-m-r}D^{\Gamma,+}_{r,v}(j,l)\\
        &+D^{\Gamma,+}_{1,1}(j,l)\sum_{m=0}^{n-1}\sum_{r=0}^{n-1-m}(m+1)D^{\Gamma,-}_{m,1}(j,l)\Lambda^\Gamma_{j,l,n-1-m-r}D^{\Gamma,+}_{r,v}(j,l)\\
        &+\sum_{k=2}^{n}D^{\Gamma,+}_{k,1}(j,l)\sum_{m=0}^{n-k}\sum_{r=0}^{n-k-m}(m+1)D^{\Gamma,-}_{m,1}(j,l)\Lambda^\Gamma_{j,l,n-k-m-r}D^{\Gamma,+}_{r,v}(j,l)\\
        &+2D^{\Gamma,+}_{1,1}(j,l)n\sum_{u=0}^{n}\Lambda^\Gamma_{j,l,n-u}D^{\Gamma,+}_{u,v-1}(j,l)\\
        &+2\sum_{k=2}^{n}D^{\Gamma,+}_{k,1}(j,l)(n-k+1)\sum_{u=0}^{n-k+1}\Lambda^\Gamma_{j,l,n-k-u+1}D^{\Gamma,+}_{u,v-1}(j,l)\\
        &-2D^{\Gamma,+}_{1,1}(j,l)\sum_{m=0}^{n-1}\sum_{r=0}^{n-1-m}(m+1)D^{\Gamma,-}_{m,1}(j,l)\Lambda^\Gamma_{j,l,n-1-m-r}D^{\Gamma,+}_{r,v}(j,l)\\
        &-2\sum_{k=2}^{n}D^{\Gamma,+}_{k,1}(j,l)\sum_{m=0}^{n-k}\sum_{r=0}^{n-k-m}(m+1)D^{\Gamma,-}_{m,1}(j,l)\Lambda^\Gamma_{j,l,n-k-m-r}D^{\Gamma,+}_{r,v}(j,l)\\
        &-\sum_{k=2}^{n}D^{\Gamma,+}_{k,1}(j,l)(n-k+1)\sum_{u=0}^{n-k+1}\Lambda^\Gamma_{j,l,n-k+1-u}D^{\Gamma,+}_{u,v-1}(j,l)\\
        &+\sum_{k=2}^{n}D^{\Gamma,+}_{k,1}(j,l)\sum_{m=0}^{n-k}\sum_{r=0}^{n-k-m}(m+1)D^{\Gamma,-}_{m,1}(j,l)\Lambda^\Gamma_{j,l,n-k-m-r}D^{\Gamma,+}_{r,v}(j,l)\\
        &=-(n+1)D^{\Gamma,+}_{0,1}(j,l)\sum_{u=0}^{n+1}\Lambda^\Gamma_{j,l,n+1-u}D^{\Gamma,+}_{u,v-1}(j,l)+nD^{\Gamma,+}_{1,1}(j,l)\sum_{u=0}^{n}\Lambda^\Gamma_{j,l,n-u}D^{\Gamma,+}_{u,v-1}(j,l)\\
        &+\sum_{m=0}^{n}\sum_{r=0}^{n-m}(m+1)D^{\Gamma,+}_{0,1}(j,l)D^{\Gamma,-}_{m,1}(j,l)\Lambda^\Gamma_{j,l,n-m-r}D^{\Gamma,+}_{r,v}(j,l)\\
        &-\sum_{m=0}^{n-1}\sum_{r=0}^{n-1-m}(m+1)D^{\Gamma,+}_{1,1}(j,l)D^{\Gamma,-}_{m,1}(j,l)\Lambda^\Gamma_{j,l,n-1-m-r}D^{\Gamma,+}_{r,v}(j,l)\\
        &=-(n+1)\sum_{u=0}^{n+1}\sum_{i=0}^{n+1-u}(i+1)\Lambda^\Gamma_{j,l,n+1-u-i}D^{\Gamma,+}_{i,1}(j,l)D^{\Gamma,+}_{u,v-1}(j,l)\\
        &+n\sum_{u=0}^{n}\sum_{i=0}^{n-u}(i+1)\Lambda^\Gamma_{j,l,n-u-i}D^{\Gamma,+}_{i+1,1}(j,l)D^{\Gamma,+}_{u,v-1}(j,l)\\
        &+\sum_{m=0}^{n}\sum_{r=0}^{n-m}(m+1)D^{\Gamma,-}_{m,1}(j,l)D^{\Gamma,+}_{0,1}(j,l)\Lambda^\Gamma_{j,l,n-m-r}D^{\Gamma,+}_{r,v}(j,l)\\
        &+\sum_{m=0}^{n}\sum_{r=0}^{n-m}(m+1)p^\Gamma_{m+1}(j,l)\Lambda^\Gamma_{j,l,n-m-r}D^{\Gamma,+}_{r,v}(j,l)\\
        &-\sum_{m=0}^{n-1}\sum_{r=0}^{n-m-1}(m+1)D^{\Gamma,-}_{m,1}(j,l)D^{\Gamma,+}_{1,1}(j,l)\Lambda^\Gamma_{j,l,n-m-r-1}D^{\Gamma,+}_{r,v}(j,l)\\
        &-\sum_{m=0}^{n-1}\sum_{r=0}^{n-1-m}(m+1)p^\Gamma_{m+2}(j,l)\Lambda^\Gamma_{j,l,n-1-m-r}D^{\Gamma,+}_{r,v}(j,l)\textnormal{ by Propositions }\ref{Du1Lambda}\textnormal{ and }\ref{pnewD}\\
        &=-(n+1)\sum_{u=0}^{n+1}\sum_{i=0}^{n+1-u}(i+1)\Lambda^\Gamma_{j,l,n+1-u-i}D^{\Gamma,+}_{i,1}(j,l)D^{\Gamma,+}_{u,v-1}(j,l)\\
        &+n\sum_{u=0}^{n+1}\sum_{i=0}^{n-u+1}i\Lambda^\Gamma_{j,l,n-u-i+1}D^{\Gamma,+}_{i,1}(j,l)D^{\Gamma,+}_{u,v-1}(j,l)\\
        &+\sum_{m=0}^{n}\sum_{r=0}^{n-m}(m+1)D^{\Gamma,-}_{m,1}(j,l)D^{\Gamma,+}_{0,1}(j,l)\Lambda^\Gamma_{j,l,n-m-r}D^{\Gamma,+}_{r,v}(j,l)\\
        &+\sum_{m=0}^{n}\sum_{r=0}^{n-m}(m+1)p^\Gamma_{m+1}(j,l)\Lambda^\Gamma_{j,l,n-m-r}D^{\Gamma,+}_{r,v}(j,l)\\
        &-\sum_{m=0}^{n-1}\sum_{r=0}^{n-m-1}(m+1)D^{\Gamma,-}_{m,1}(j,l)D^{\Gamma,+}_{1,1}(j,l)\Lambda^\Gamma_{j,l,n-m-r-1}D^{\Gamma,+}_{r,v}(j,l)\\
        &-\sum_{m=0}^{n}\sum_{r=0}^{n-m}mp^\Gamma_{m+1}(j,l)\Lambda^\Gamma_{j,l,n-m-r}D^{\Gamma,+}_{r,v}(j,l)\\
        &=-\sum_{u=0}^{n+1}\sum_{i=0}^{n+1-u}(n+i+1)\Lambda^\Gamma_{j,l,n+1-u-i}D^{\Gamma,+}_{i,1}(j,l)D^{\Gamma,+}_{u,v-1}(j,l)\\
        &+\sum_{m=0}^{n}\sum_{r=0}^{n-m}(m+1)D^{\Gamma,-}_{m,1}(j,l)D^{\Gamma,+}_{0,1}(j,l)\Lambda^\Gamma_{j,l,n-m-r}D^{\Gamma,+}_{r,v}(j,l)+\sum_{m=0}^{n}\sum_{r=0}^{n-m}p^\Gamma_{m+1}(j,l)\Lambda^\Gamma_{j,l,n-m-r}D^{\Gamma,+}_{r,v}(j,l)\\
        &-\sum_{m=0}^{n-1}\sum_{r=0}^{n-m-1}(m+1)D^{\Gamma,-}_{m,1}(j,l)D^{\Gamma,+}_{1,1}(j,l)\Lambda^\Gamma_{j,l,n-m-r-1}D^{\Gamma,+}_{r,v}(j,l)\\
        &=-\sum_{i=0}^{n+1}\sum_{u=0}^{n+1-i}(n+i+1)\Lambda^\Gamma_{j,l,n+1-u-i}D^{\Gamma,+}_{i,1}(j,l)D^{\Gamma,+}_{u,v-1}(j,l)\\
        &+\sum_{m=0}^{n}\sum_{r=0}^{n-m}(m+1)D^{\Gamma,-}_{m,1}(j,l)D^{\Gamma,+}_{0,1}(j,l)\Lambda^\Gamma_{j,l,n-m-r}D^{\Gamma,+}_{r,v}(j,l)+\sum_{r=0}^{n}\sum_{m=0}^{n-r}p^\Gamma_{m+1}(j,l)\Lambda^\Gamma_{j,l,n-m-r}D^{\Gamma,+}_{r,v}(j,l)\\
        &-\sum_{m=0}^{n-1}\sum_{r=0}^{n-m-1}(m+1)D^{\Gamma,-}_{m,1}(j,l)D^{\Gamma,+}_{1,1}(j,l)\Lambda^\Gamma_{j,l,n-m-r-1}D^{\Gamma,+}_{r,v}(j,l)\\
        &=-\sum_{i=0}^{n+1}\sum_{u=i}^{n+1}(n+i+1)\Lambda^\Gamma_{j,l,n+1-u}D^{\Gamma,+}_{i,1}(j,l)D^{\Gamma,+}_{u-i,v-1}(j,l)\\
        &+\sum_{m=0}^{n}\sum_{r=0}^{n-m}(m+1)D^{\Gamma,-}_{m,1}(j,l)D^{\Gamma,+}_{0,1}(j,l)\Lambda^\Gamma_{j,l,n-m-r}D^{\Gamma,+}_{r,v}(j,l)+\sum_{r=0}^{n}\sum_{m=1}^{n-r+1}p^\Gamma_{m}(j,l)\Lambda^\Gamma_{j,l,n-m-r+1}D^{\Gamma,+}_{r,v}(j,l)\\
        &-\sum_{m=0}^{n-1}\sum_{r=0}^{n-m-1}(m+1)D^{\Gamma,-}_{m,1}(j,l)D^{\Gamma,+}_{1,1}(j,l)\Lambda^\Gamma_{j,l,n-m-r-1}D^{\Gamma,+}_{r,v}(j,l)\\
        &=-nv\sum_{u=0}^{n+1}\Lambda^\Gamma_{j,l,n+1-u}D^{\Gamma,+}_{u,v}(j,l)-\sum_{u=0}^{n+1}(u+v)\Lambda^\Gamma_{j,l,n+1-u}D^{\Gamma,+}_{u,v}(j,l)\\
        &+\sum_{m=0}^{n}\sum_{r=0}^{n-m}(m+1)D^{\Gamma,-}_{m,1}(j,l)D^{\Gamma,+}_{0,1}(j,l)\Lambda^\Gamma_{j,l,n-m-r}D^{\Gamma,+}_{r,v}(j,l)-\sum_{r=0}^{n+1}(n-r+1)\Lambda^\Gamma_{j,l,n-r+1}D^{\Gamma,+}_{r,v}(j,l)\\
        &-\sum_{m=0}^{n-1}\sum_{r=0}^{n-m-1}(m+1)D^{\Gamma,-}_{m,1}(j,l)D^{\Gamma,+}_{1,1}(j,l)\Lambda^\Gamma_{j,l,n-m-r-1}D^{\Gamma,+}_{r,v}(j,l)\\
        &=-(v+1)(n+1)\sum_{u=0}^{n+1}\Lambda^\Gamma_{j,l,n+1-u}D^{\Gamma,+}_{u,v}(j,l)+\sum_{m=0}^{n}\sum_{r=0}^{n-m}(m+1)D^{\Gamma,-}_{m,1}(j,l)D^{\Gamma,+}_{0,1}(j,l)\Lambda^\Gamma_{j,l,n-m-r}D^{\Gamma,+}_{r,v}(j,l)\\
        &-\sum_{m=0}^{n-1}\sum_{r=0}^{n-m-1}(m+1)D^{\Gamma,-}_{m,1}(j,l)D^{\Gamma,+}_{1,1}(j,l)\Lambda^\Gamma_{j,l,n-m-r-1}D^{\Gamma,+}_{r,v}(j,l)\\
        &=-(v+1)(n+1)\sum_{u=0}^{n+1}\Lambda^\Gamma_{j,l,n+1-u}D^{\Gamma,+}_{u,v}(j,l)+\sum_{m=0}^{n}\sum_{r=0}^{n-m}(m+1)D^{\Gamma,-}_{m,1}(j,l)D^{\Gamma,+}_{0,1}(j,l)\Lambda^\Gamma_{j,l,n-m-r}D^{\Gamma,+}_{r,v}(j,l)\\
        &-\sum_{m=0}^{n}\sum_{r=1}^{n-m}(m+1)D^{\Gamma,-}_{m,1}(j,l)D^{\Gamma,+}_{1,1}(j,l)\Lambda^\Gamma_{j,l,n-m-r}D^{\Gamma,+}_{r-1,v}(j,l)\\
        &=-(v+1)(n+1)\sum_{u=0}^{n+1}\Lambda^\Gamma_{j,l,n+1-u}D^{\Gamma,+}_{u,v}(j,l)+\sum_{m=0}^{n}(m+1)D^{\Gamma,-}_{m,1}(j,l)D^{\Gamma,+}_{0,1}(j,l)\Lambda^\Gamma_{j,l,n-m}D^{\Gamma,+}_{0,v}(j,l)\\
        &+\sum_{m=0}^{n}\sum_{k=0}^{1}(m+1)D^{\Gamma,-}_{m,1}(j,l)D^{\Gamma,+}_{k,1}(j,l)\Lambda^\Gamma_{j,l,n-m-1}D^{\Gamma,+}_{1-k,v}(j,l)\\
        &+\sum_{m=0}^{n}\sum_{r=2}^{n-m}(m+1)D^{\Gamma,-}_{m,1}(j,l)D^{\Gamma,+}_{0,1}(j,l)\Lambda^\Gamma_{j,l,n-m-r}D^{\Gamma,+}_{r,v}(j,l)\\
        &+\sum_{m=0}^{n}\sum_{r=2}^{n-m}(m+1)D^{\Gamma,-}_{m,1}(j,l)D^{\Gamma,+}_{1,1}(j,l)\Lambda^\Gamma_{j,l,n-m-r}D^{\Gamma,+}_{r-1,v}(j,l)\\
        &+\sum_{m=0}^{n}\sum_{r=2}^{n-m}\sum_{k=2}^{r}(m+1)D^{\Gamma,-}_{m,1}(j,l)D^{\Gamma,+}_{k,1}(j,l)\Lambda^\Gamma_{j,l,n-m-r}D^{\Gamma,+}_{r-k,v}(j,l)\\
        &-2\sum_{m=0}^{n}(m+1)D^{\Gamma,-}_{m,1}(j,l)D^{\Gamma,+}_{1,1}(j,l)\Lambda^\Gamma_{j,l,n-m-1}D^{\Gamma,+}_{0,v}(j,l)\\
        &-2\sum_{m=0}^{n}\sum_{r=2}^{n-m}(m+1)D^{\Gamma,-}_{m,1}(j,l)D^{\Gamma,+}_{1,1}(j,l)\Lambda^\Gamma_{j,l,n-m-r}D^{\Gamma,+}_{r-1,v}(j,l)\\
        &-2\sum_{m=0}^{n}\sum_{r=2}^{n-m}\sum_{k=2}^{r}(m+1)D^{\Gamma,-}_{m,1}(j,l)D^{\Gamma,+}_{k,1}(j,l)\Lambda^\Gamma_{j,l,n-m-r}D^{\Gamma,+}_{r-k,v}(j,l)\\
        &+\sum_{m=0}^{n}\sum_{r=2}^{n-m}\sum_{k=2}^{r}(m+1)D^{\Gamma,-}_{m,1}(j,l)D^{\Gamma,+}_{k,1}(j,l)\Lambda^\Gamma_{j,l,n-m-r}D^{\Gamma,+}_{r-k,v}(j,l)\\
        &=-(v+1)(n+1)\sum_{u=0}^{n+1}\Lambda^\Gamma_{j,l,n+1-u}D^{\Gamma,+}_{u,v}(j,l)\\
        &+\sum_{m=0}^{n}\sum_{r=0}^{n-m}\sum_{k=0}^{r}(m+1)D^{\Gamma,-}_{m,1}(j,l)D^{\Gamma,+}_{k,1}(j,l)\Lambda^\Gamma_{j,l,n-m-r}D^{\Gamma,+}_{r-k,v}(j,l)\\
        &-2\sum_{m=0}^{n}\sum_{r=1}^{n-m}\sum_{k=1}^{r}(m+1)D^{\Gamma,-}_{m,1}(j,l)D^{\Gamma,+}_{k,1}(j,l)\Lambda^\Gamma_{j,l,n-m-r}D^{\Gamma,+}_{r-k,v}(j,l)\\
        &+\sum_{m=0}^{n}\sum_{r=2}^{n-m}\sum_{k=2}^{r}(m+1)D^{\Gamma,-}_{m,1}(j,l)D^{\Gamma,+}_{k,1}(j,l)\Lambda^\Gamma_{j,l,n-m-r}D^{\Gamma,+}_{r-k,v}(j,l)\\
        &=-(v+1)(n+1)\sum_{u=0}^{n+1}\Lambda^\Gamma_{j,l,n+1-u}D^{\Gamma,+}_{u,v}(j,l)\\
        &+\sum_{m=0}^{n}\sum_{r=0}^{n-m}\sum_{k=0}^{r}(m+1)D^{\Gamma,-}_{m,1}(j,l)D^{\Gamma,+}_{k,1}(j,l)\Lambda^\Gamma_{j,l,n-m-r}D^{\Gamma,+}_{r-k,v}(j,l)\\
        &-2\sum_{m=0}^{n}\sum_{r=0}^{n-m}\sum_{k=0}^{r}(m+1)D^{\Gamma,-}_{m,1}(j,l)D^{\Gamma,+}_{k+1,1}(j,l)\Lambda^\Gamma_{j,l,n-m-r-1}D^{\Gamma,+}_{r-k,v}(j,l)\\
        &+\sum_{m=0}^{n}\sum_{r=0}^{n-m}\sum_{k=0}^{r}(m+1)D^{\Gamma,-}_{m,1}(j,l)D^{\Gamma,+}_{k+2,1}(j,l)\Lambda^\Gamma_{j,l,n-m-r-2}D^{\Gamma,+}_{r-k,v}(j,l)\\
        &=-(v+1)(n+1)\sum_{u=0}^{n+1}\Lambda^\Gamma_{j,l,n+1-u}D^{\Gamma,+}_{u,v}(j,l)\\
        &+\sum_{m=0}^{n}\sum_{r=0}^{n-m}\sum_{k=0}^{r}(m+1)D^{\Gamma,-}_{m,1}(j,l)\Big(D^{\Gamma,+}_{k,1}(j,l)\Lambda^\Gamma_{j,l,n-m-r}\\
        &-2D^{\Gamma,+}_{k+1,1}(j,l)\Lambda^\Gamma_{j,l,n-m-r-1}+D^{\Gamma,+}_{k+2,1}(j,l)\Lambda^\Gamma_{j,l,n-m-r-2}\Big)D^{\Gamma,+}_{r-k,v}(j,l)\\
        &=-(v+1)(n+1)\sum_{u=0}^{n+1}\Lambda^\Gamma_{j,l,n+1-u}D^{\Gamma,+}_{u,v}(j,l)+\sum_{m=0}^{n}\sum_{r=0}^{n-m}\sum_{k=0}^{r}(m+1)D^{\Gamma,-}_{m,1}(j,l)\Lambda^\Gamma_{j,l,n-m-r}D^{\Gamma,+}_{k,1}(j,l)D^{\Gamma,+}_{r-k,v}(j,l)\\
        &\textnormal{ by Proposition }\ref{LambdaD+}\\
        &=-(v+1)(n+1)\sum_{u=0}^{n+1}\Lambda^\Gamma_{j,l,n+1-u}D^{\Gamma,+}_{u,v}(j,l)+(v+1)\sum_{m=0}^{n}\sum_{r=0}^{n-m}(m+1)D^{\Gamma,-}_{m,1}(j,l)\Lambda^\Gamma_{j,l,n-m-r}D^{\Gamma,+}_{r,v+1}(j,l)
    \end{align*}
\end{proof}
  
Now, we prove Proposition \ref{straightening} (7):
\begin{proof}
    We proceed by induction on $s$. The case $s=0$ holds, because $D_{0,r}^{\Gamma,+}(j,l)=\left(x_j^{\Gamma,+}\right)^{(r)}$. Assume the formula is true for some $s\in{\mathbb Z}_+$. Then we have
    $$(s+1)\left(x_j^{\Gamma,+}\right)^{(r)}\left(x_l^{\Gamma,-}\right)^{(s+1)}$$
    \begin{align*}
        &=\left(x_j^{\Gamma,+}\right)^{(r)}\left(x_l^{\Gamma,-}\right)^{(s)}\left(x_l^{\Gamma,-}\right)\\
        &=\sum_{\substack{m,n,q\in{\mathbb Z}_+\\m+n+q\leq\min\{r,s\}}}(-1)^{m+n+q}D^{\Gamma,-}_{m,s-m-n-q}(j,l)\Lambda^\Gamma_{j,l,n}D^{\Gamma,+}_{q,r-m-n-q}(j,l)\left(x_l^{\Gamma,-}\right)\\
        &\textnormal{by the induction hypothesis}\\
        &=\sum_{\substack{m,i,q\in{\mathbb Z}_+\\m+i+q\leq\min\{r,s\}}}(-1)^{m+i+q}D^{\Gamma,-}_{m,s-m-i-q}(j,l)\Lambda^\Gamma_{j,l,i}D^{\Gamma,+}_{q,r-m-i-q}(j,l)\left(x_l^{\Gamma,-}\right)\\
        &=\sum_{\substack{m,i,n\in{\mathbb Z}_+\\m+n\leq\min\{r,s\}}}(-1)^{m+n}D^{\Gamma,-}_{m,s-m-n}(j,l)\Lambda^\Gamma_{j,l,i}D^{\Gamma,+}_{n-i,r-m-n}(j,l)\left(x_l^{\Gamma,-}\right)\\
        &=\sum_{\substack{m,n\in{\mathbb Z}_+\\m+n\leq\min\{r,s\}}}(-1)^{m+n}D^{\Gamma,-}_{m,s-m-n}(j,l)\sum_{i=0}^n\Lambda^\Gamma_{j,l,i}D^{\Gamma,+}_{n-i,r-m-n}(j,l)\left(x_l^{\Gamma,-}\right)\\
        &=\sum_{\substack{m,n\in{\mathbb Z}_+\\m+n\leq\min\{r,s\}}}(-1)^{m+n}D^{\Gamma,-}_{m,s-m-n}(j,l)\bigg(-(n+1)\sum_{u=0}^{n+1}\Lambda^\Gamma_{j,l,n+1-u}D^{\Gamma,+}_{u,r-m-n-1}(j,l)\\
        &+\sum_{i=0}^{n}\sum_{k=0}^{n-i}(i+1)D^{\Gamma,-}_{i,1}(j,l)\Lambda^\Gamma_{j,l,n-i-k}D^{\Gamma,+}_{k,r-m-n}(j,l)\bigg)\textnormal{ by Lemma }\ref{LambdaD+x-}\\
        &=\sum_{\substack{m,n\in{\mathbb Z}_+\\m+n\leq\min\{r,s\}}}(-1)^{m+n+1}D^{\Gamma,-}_{m,s-m-n}(j,l)(n+1)\sum_{u=0}^{n+1}\Lambda^\Gamma_{j,l,n+1-u}D^{\Gamma,+}_{u,r-m-n-1}(j,l)\\
        &+\sum_{\substack{m,n\in{\mathbb Z}_+\\m+n\leq\min\{r,s\}}}(-1)^{m+n}\sum_{i=0}^{n}(i+1)D^{\Gamma,-}_{i,1}(j,l)D^{\Gamma,-}_{m,s-m-n}(j,l)\sum_{k=0}^{n-i}\Lambda^\Gamma_{j,l,n-i-k}D^{\Gamma,+}_{k,r-m-n}(j,l)\\
        &=\sum_{\substack{m,n\in{\mathbb Z}_+\\m+n\leq\min\{r,s\}}}(-1)^{m+n+1}D^{\Gamma,-}_{m,s-m-n}(j,l)(n+1)\sum_{u=0}^{n+1}\Lambda^\Gamma_{j,l,u}D^{\Gamma,+}_{n+1-u,r-m-n-1}(j,l)\\
        &+\sum_{\substack{m,n\in{\mathbb Z}_+\\m+n\leq\min\{r,s\}}}(-1)^{m+n}\sum_{i=0}^{n}(i+1)D^{\Gamma,-}_{i,1}(j,l)D^{\Gamma,-}_{m,s-m-n}(j,l)\sum_{k=0}^{n-i}\Lambda^\Gamma_{j,l,k}D^{\Gamma,+}_{n-i-k,r-m-n}(j,l)\\
        &=\sum_{\substack{m,n,u\in{\mathbb Z}_+\\m+n\leq\min\{r,s\}}}(-1)^{m+n+1}(n+1)D^{\Gamma,-}_{m,s-m-n}(j,l)\Lambda^\Gamma_{j,l,u}D^{\Gamma,+}_{n+1-u,r-m-n-1}(j,l)\\
        &+\sum_{\substack{m,n\in{\mathbb Z}_+\\m+n\leq\min\{r,s\}}}(-1)^{m+n}\sum_{k=0}^{n}\sum_{i=0}^{n-k}(i+1)D^{\Gamma,-}_{i,1}(j,l)D^{\Gamma,-}_{m,s-m-n}(j,l)\Lambda^\Gamma_{j,l,k}D^{\Gamma,+}_{n-i-k,r-m-n}(j,l)\\
        &=\sum_{\substack{m,u,q\in{\mathbb Z}_+\\m+q+u\leq\min\{r,s\}}}(-1)^{m+q+u+1}(q+u+1)D^{\Gamma,-}_{m,s-m-q-u}(j,l)\Lambda^\Gamma_{j,l,u}D^{\Gamma,+}_{q+1,r-m-q-u-1}(j,l)\\
        &+\sum_{\substack{m,q,k\in{\mathbb Z}_+\\m+q+k\leq\min\{r,s\}}}(-1)^{m+q+k}\sum_{i=0}^{q}(i+1)D^{\Gamma,-}_{i,1}(j,l)D^{\Gamma,-}_{m,s-m-q-k}(j,l)\Lambda^\Gamma_{j,l,k}D^{\Gamma,+}_{q-i,r-m-q-k}(j,l)\\
        &=\sum_{\substack{m,u,q\in{\mathbb Z}_+\\1\leq m+q+u\leq\min\{r,s\}+1}}(-1)^{m+q+u}(q+u)D^{\Gamma,-}_{m,s+1-m-q-u}(j,l)\Lambda^\Gamma_{j,l,u}D^{\Gamma,+}_{q,r-m-q-u}(j,l)\\
        &+\sum_{i=0}^{\min\{r,s\}}\sum_{\substack{m,q,k\in{\mathbb Z}_+\\m+q+k\leq\min\{r,s\}}}(-1)^{m+q+k}(i+1)D^{\Gamma,-}_{i,1}(j,l)D^{\Gamma,-}_{m,s-m-q-k}(j,l)\Lambda^\Gamma_{j,l,k}D^{\Gamma,+}_{q-i,r-m-q-k}(j,l)\\
        &=\sum_{\substack{m,u,q\in{\mathbb Z}_+\\1\leq m+q+u\leq\min\{r,s+1\}}}(-1)^{m+q+u}(q+u)D^{\Gamma,-}_{m,s+1-m-q-u}(j,l)\Lambda^\Gamma_{j,l,u}D^{\Gamma,+}_{q,r-m-q-u}(j,l)\\
        &+\sum_{i=0}^{\min\{r,s\}}\sum_{\substack{m,q,k\in{\mathbb Z}_+\\m+q+k\leq\min\{r,s\}}}(-1)^{m+q+k}(i+1)D^{\Gamma,-}_{i,1}(j,l)D^{\Gamma,-}_{m-i,s-m-q-k}(j,l)\Lambda^\Gamma_{j,l,k}D^{\Gamma,+}_{q,r-m-q-k}(j,l)\\
        &=\sum_{\substack{m,u,q\in{\mathbb Z}_+\\1\leq m+q+u\leq\min\{r,s+1\}}}(-1)^{m+q+u}(q+u)D^{\Gamma,-}_{m,s+1-m-q-u}(j,l)\Lambda^\Gamma_{j,l,u}D^{\Gamma,+}_{q,r-m-q-u}(j,l)\\
        &+\sum_{\substack{m,q,k\in{\mathbb Z}_+\\m+q+k\leq\min\{r,s\}}}(-1)^{m+q+k}\sum_{i=0}^{m}(i+1)D^{\Gamma,-}_{i,1}(j,l)D^{\Gamma,-}_{m-i,s-m-q-k}(j,l)\Lambda^\Gamma_{j,l,k}D^{\Gamma,+}_{q,r-m-q-k}(j,l)\\
        &=\sum_{\substack{m,u,q\in{\mathbb Z}_+\\1\leq m+q+u\leq\min\{r,s+1\}}}(-1)^{m+q+u}(q+u)D^{\Gamma,-}_{m,s+1-m-q-u}(j,l)\Lambda^\Gamma_{j,l,u}D^{\Gamma,+}_{q,r-m-q-u}(j,l)\\
        &+\sum_{\substack{m,q,k\in{\mathbb Z}_+\\m+q+k\leq\min\{r,s\}}}(-1)^{m+q+k}(s+1-q-k)D^{\Gamma,-}_{m,s+1-m-q-k}(j,l)\Lambda^\Gamma_{j,l,k}D^{\Gamma,+}_{q,r-m-q-k}(j,l)\\
        &\textnormal{ by Proposition }\ref{uD}(ii)\\
        &=\sum_{\substack{m,n,q\in{\mathbb Z}_+\\m+n+q\leq\min\{r,s+1\}}}(-1)^{m+n+q}(q+n)D^{\Gamma,-}_{m,s+1-m-n-q}(j,l)\Lambda^\Gamma_{j,l,n}D^{\Gamma,+}_{q,r-m-n-q}(j,l)\\
        &+\sum_{\substack{m,n,q\in{\mathbb Z}_+\\m+n+q\leq\min\{r,s+1\}}}(-1)^{m+n+q}(s+1-q-n)D^{\Gamma,-}_{m,s+1-m-n-q}(j,l)\Lambda^\Gamma_{j,l,n}D^{\Gamma,+}_{q,r-m-n-q}(j,l)\\
        &=(s+1)\sum_{\substack{m,n,q\in{\mathbb Z}_+\\m+n+q\leq\min\{r,s+1\}}}(-1)^{m+n+q}D^{\Gamma,-}_{m,s+1-m-n-q}(j,l)\Lambda^\Gamma_{j,l,n}D^{\Gamma,+}_{q,r-m-n-q}(j,l)
    \end{align*}
\end{proof}

\details{
\section*{Further details of hidden proofs}

\subsection*{Proof of Proposition \ref{Duv}}

Recall that we are trying to show that for all $j,l,u,v\in\mathbb N$
\begin{align*}
    D^{\Gamma,\pm}_{u,v}(j,l)
    &=\sum_{\substack{k_0,\ldots,k_u\in{\mathbb Z}_+\\k_0+\dots k_u=v\\k_1+2k_2+\dots+uk_u=u}}\left(D^{\Gamma,\pm}_{0,1}(j,l)\right)^{(k_0)}\dots\left(D^{\Gamma,\pm}_{u,1}(j,l)\right)^{(k_u)}\\
    &=\left(\left(\sum_{m\geq0}D^{\Gamma,\pm}_{m,1}(j,l)w^{m+1}\right)^{(v)}\right)_{u+v}
\end{align*}
\begin{proof} 
    We will prove the first equation by induction on $v$. The second equation then follows by the Multinomial Theorem. If $v=1$ the terms of the sum on the right hand side are only nonzero when $k_u=1$ and all other $k_i=0$. In this case the right hand side equals $D^{\Gamma,\pm}_{u,1}(j,l)$. Assume that the formula holds for some $v\in\mathbb N$. Then
    \begin{align*}
        D^{\Gamma,\pm}_{u,v+1}(j,l)
        &=\frac{1}{v+1}\sum_{i=0}^u D^{\Gamma,\pm}_{i,1}(j,l)D^{\Gamma,\pm}_{u-i,v}(j,l)\\
        &=\frac{1}{v+1}\sum_{i=0}^{u-1}D^{\Gamma,\pm}_{i,1}(j,l)D^{\Gamma,\pm}_{u-i,v}(j,l)+\frac{1}{v+1}D^{\Gamma,\pm}_{u,1}(j,l)D^{\Gamma,\pm}_{0,v}(j,l)\\
        &=\frac{1}{v+1}\sum_{i=0}^{u-1}D^{\Gamma,\pm}_{i,1}(j,l)\sum_{\substack{k_0,\ldots,k_{u-i}\in{\mathbb Z}_+\\k_0+\dots+k_{u-i}=v\\k_1+2k_2+\dots+(u-i)k_{u-i}=u-i}}\left(D^{\Gamma,\pm}_{0,1}(j,l)\right)^{(k_0)}\dots\left(D^{\Gamma,\pm}_{u-i,1}(j,l)\right)^{(k_{u-i})}\\
        &+\frac{1}{v+1}D^{\Gamma,\pm}_{u,1}(j,l)D^{\Gamma,\pm}_{0,v}(j,l)\textnormal{ by the induction hypothesis}\\
        &=\frac{1}{v+1}\sum_{i=0}^{u-1}\sum_{\substack{k_0,\ldots,k_{u-i}\in{\mathbb Z}_+\\k_0+\dots+k_{u-i}=v\\k_1+2k_2+\dots+(u-i)k_{u-i}=u-i}}(k_i+1)\left(D^{\Gamma,\pm}_{0,1}(j,l)\right)^{(k_0)}\dots\left(D^{\Gamma,\pm}_{i,1}(j,l)\right)^{(k_i+1)}\dots\\
        &\cdots\left(D^{\Gamma,\pm}_{u-i,1}(j,l)\right)^{(k_{u-i})}+\frac{1}{v+1}D^{\Gamma,\pm}_{u,1}(j,l)D^{\Gamma,\pm}_{0,v}(j,l)\\
        &=\frac{1}{v+1}\sum_{i=0}^{u-1}\sum_{\substack{k_0,\ldots,k_{u-i}\in{\mathbb Z}_+\\k_0+\dots+k_{u-i}=v+1\\k_1+2k_2+\dots+(u-i)k_{u-i}=u}}k_i\left(D^{\Gamma,\pm}_{0,1}(j,l)\right)^{(k_0)}\dots\left(D^{\Gamma,\pm}_{u-i,1}(j,l)\right)^{(k_{u-i})}\\
        &+\frac{1}{v+1}\left(D^{\Gamma,\pm}_{0,1}(j,l)\right)^{(v)}D^{\Gamma,\pm}_{u,1}(j,l)\\
        &=\frac{1}{v+1}\sum_{i=0}^{u}\sum_{\substack{k_0,\ldots,k_{u}\in{\mathbb Z}_+\\k_0+\dots+k_{u}=v+1\\k_1+2k_2+\dots+uk_{u}=u}}k_i\left(D^{\Gamma,\pm}_{0,1}(j,l)\right)^{(k_0)}\dots\left(D^{\Gamma,\pm}_{u,1}(j,l)\right)^{(k_{u})}\\
        &=\sum_{\substack{k_0,\ldots,k_{u}\in{\mathbb Z}_+\\k_0+\dots+k_{u}=v+1\\k_1+2k_2+\dots+uk_{u}=u}}\left(D^{\Gamma,\pm}_{0,1}(j,l)\right)^{(k_0)}\dots\left(D^{\Gamma,\pm}_{u,1}(j,l)\right)^{(k_{u})}\\
    \end{align*}
    \end{proof} 

\subsection*{Proof of Proposition \ref{Du1}\eqref{D+u1}}
Recall that the statement we wish to prove is
for all $u\in{\mathbb Z}_+$ and $j,l\in\mathbb N$
    \begin{align*}
        D^{\Gamma,+}_{u,1}(j,l)
        &=\sum_{k=0}^{\left\lfloor\frac{u-1}{2}\right\rfloor}\sum_{i=0}^{u+1}(-1)^{k+i}\binom{u}{k}\binom{u+1}{i}\left(x_{(u+1-2i)j+(u-2k)l}^{\Gamma,+}\right)\nonumber\\
        &+((u+1)\mod2)\sum_{i=0}^{\frac{u}{2}}(-1)^{\frac{u}{2}+i}\binom{u}{\frac{u}{2}}\binom{u+1}{i}\left(x_{(u+1-2i)j}^{\Gamma,+}\right)
    \end{align*}
    We proceed by induction on $u$. The base case $u=0$ is clear. 
      \begin{align*}
        D^{\Gamma,+}_{u+1,1}(j,l)
        &=\frac{1}{2}\left[D^{\Gamma,+}_{u,1}(j,l),\Lambda_{j,l,1}^\Gamma\right]\\
        &=\frac{1}{2}\sum_{k=0}^{\left\lfloor\frac{u-1}{2}\right\rfloor}\sum_{i=0}^{u+1}(-1)^{k+i}\binom{u}{k}\binom{u+1}{i}\left[\left(x_{(u+1-2i)j+(u-2k)l}^{\Gamma,+}\right),\Lambda_{j,l,1}^\Gamma\right]\\
        &+\frac{1}{2}((u+1)\mod2)\sum_{i=0}^{\frac{u}{2}}(-1)^{\frac{u}{2}+i}\binom{u}{\frac{u}{2}}\binom{u+1}{i}\left[\left(x_{(u+1-2i)j}^{\Gamma,+}\right),\Lambda_{j,l,1}^\Gamma\right]\\
        &\textnormal{by the induction hypothesis}\\
        &=\sum_{k=0}^{\left\lfloor\frac{u-1}{2}\right\rfloor}\sum_{i=0}^{u+1}(-1)^{k+i}\binom{u}{k}\binom{u+1}{i}\bigg(\left(x_{(u+2-2i)j+(u+1-2k)l}^{\Gamma,+}\right)+\left(x_{(u-2i)j+(u-1-2k)l}^{\Gamma,+}\right)\\
        &-\left(x_{(u+2-2i)j+(u-1-2k)l}^{\Gamma,+}\right)-\left(x_{(u-2i)j+(u+1-2k)l}^{\Gamma,+}\right)\bigg)\\
        &+((u+1)\mod2)\sum_{i=0}^{\frac{u}{2}}(-1)^{\frac{u}{2}+i}\binom{u}{\frac{u}{2}}\binom{u+1}{i}\bigg(\left(x_{(u+2-2i)j+l}^{\Gamma,+}\right)-\left(x_{-(u-2i)j+l}^{\Gamma,+}\right)+\left(x_{-(u+2-2i)j+l}^{\Gamma,+}\right)\\
        & -\left(x_{(u-2i)j+l}^{\Gamma,+}\right)\bigg)\textnormal{ by }\eqref{xk+Lambda_1}\\
        &=\sum_{k=0}^{\left\lfloor\frac{u-1}{2}\right\rfloor}\sum_{i=0}^{u+1}(-1)^{k+i}\binom{u}{k}\binom{u+1}{i}\left(x_{(u+2-2i)j+(u+1-2k)l}^{\Gamma,+}\right)\\
        &+\sum_{k=0}^{\left\lfloor\frac{u-1}{2}\right\rfloor}\sum_{i=0}^{u+1}(-1)^{k+i}\binom{u}{k}\binom{u+1}{i}\left(x_{(u-2i)j+(u-1-2k)l}^{\Gamma,+}\right)\\
        &-\sum_{k=0}^{\left\lfloor\frac{u-1}{2}\right\rfloor}\sum_{i=0}^{u+1}(-1)^{k+i}\binom{u}{k}\binom{u+1}{i}\left(x_{(u+2-2i)j+(u-1-2k)l}^{\Gamma,+}\right)\\
        &-\sum_{k=0}^{\left\lfloor\frac{u-1}{2}\right\rfloor}\sum_{i=0}^{u+1}(-1)^{k+i}\binom{u}{k}\binom{u+1}{i}\left(x_{(u-2i)j+(u+1-2k)l}^{\Gamma,+}\right)\\
        &+((u+1)\mod2)\Bigg(\sum_{i=0}^{\frac{u}{2}}(-1)^{\frac{u}{2}+i}\binom{u}{\frac{u}{2}}\binom{u+1}{i}\left(x_{(u+2-2i)j+l}^{\Gamma,+}\right)\\
        &-\sum_{i=0}^{\frac{u}{2}}(-1)^{\frac{u}{2}+i}\binom{u}{\frac{u}{2}}\binom{u+1}{i}\left(x_{-(u-2i)j+l}^{\Gamma,+}\right)+\sum_{i=0}^{\frac{u}{2}}(-1)^{\frac{u}{2}+i}\binom{u}{\frac{u}{2}}\binom{u+1}{i}\left(x_{-(u+2-2i)j+l}^{\Gamma,+}\right)\\
        &-\sum_{i=0}^{\frac{u}{2}}(-1)^{\frac{u}{2}+i}\binom{u}{\frac{u}{2}}\binom{u+1}{i}\left(x_{(u-2i)j+l}^{\Gamma,+}\right)\Bigg)\\
        &=\sum_{k=0}^{\left\lfloor\frac{u-1}{2}\right\rfloor}\sum_{i=0}^{u+1}(-1)^{k+i}\binom{u}{k}\binom{u+1}{i}\left(x_{(u+2-2i)j+(u+1-2k)l}^{\Gamma,+}\right)\\
        &+\sum_{k=1}^{\left\lfloor\frac{u+1}{2}\right\rfloor}\sum_{i=1}^{u+2}(-1)^{k+i-2}\binom{u}{k-1}\binom{u+1}{i-1}\left(x_{(u+2-2i)j+(u+1-2k)l}^{\Gamma,+}\right)\\
        &-\sum_{k=1}^{\left\lfloor\frac{u+1}{2}\right\rfloor}\sum_{i=0}^{u+1}(-1)^{k-1+i}\binom{u}{k-1}\binom{u+1}{i}\left(x_{(u+2-2i)j+(u+1-2k)l}^{\Gamma,+}\right)\\
        &-\sum_{k=0}^{\left\lfloor\frac{u-1}{2}\right\rfloor}\sum_{i=1}^{u+2}(-1)^{k+i-1}\binom{u}{k}\binom{u+1}{i-1}\left(x_{(u+2-2i)j+(u+1-2k)l}^{\Gamma,+}\right)\\
        &+((u+1)\mod2)\Bigg(\sum_{i=0}^{\frac{u}{2}}(-1)^{\frac{u}{2}+i}\binom{u}{\frac{u}{2}}\binom{u+1}{i}\left(x_{(u+2-2i)j+l}^{\Gamma,+}\right)\\
        &-\sum_{i=1}^{\frac{u+2}{2}}(-1)^{\frac{u}{2}+i-1}\binom{u}{\frac{u}{2}}\binom{u+1}{i-1}\left(x_{-(u+2-2i)j+l}^{\Gamma,+}\right)+\sum_{i=0}^{\frac{u}{2}}(-1)^{\frac{u}{2}+i}\binom{u}{\frac{u}{2}}\binom{u+1}{i}\left(x_{-(u+2-2i)j+l}^{\Gamma,+}\right)\\
        &-\sum_{i=1}^{\frac{u+2}{2}}(-1)^{\frac{u}{2}+i-1}\binom{u}{\frac{u}{2}}\binom{u+1}{i-1}\left(x_{(u+2-2i)j+l}^{\Gamma,+}\right)\Bigg)\\
        &=\sum_{k=0}^{\left\lfloor\frac{u-1}{2}\right\rfloor}\sum_{i=0}^{u+1}(-1)^{k+i}\binom{u}{k}\binom{u+1}{i}\left(x_{(u+2-2i)j+(u+1-2k)l}^{\Gamma,+}\right)\\
        &+\sum_{k=1}^{\left\lfloor\frac{u+1}{2}\right\rfloor}\sum_{i=1}^{u+2}(-1)^{k+i}\binom{u}{k-1}\binom{u+1}{i-1}\left(x_{(u+2-2i)j+(u+1-2k)l}^{\Gamma,+}\right)\\
        &+\sum_{k=1}^{\left\lfloor\frac{u+1}{2}\right\rfloor}\sum_{i=0}^{u+1}(-1)^{k+i}\binom{u}{k-1}\binom{u+1}{i}\left(x_{(u+2-2i)j+(u+1-2k)l}^{\Gamma,+}\right)\\
        &+\sum_{k=0}^{\left\lfloor\frac{u-1}{2}\right\rfloor}\sum_{i=1}^{u+2}(-1)^{k+i}\binom{u}{k}\binom{u+1}{i-1}\left(x_{(u+2-2i)j+(u+1-2k)l}^{\Gamma,+}\right)\\
        &+((u+1)\mod2)\Bigg(\sum_{i=0}^{\frac{u}{2}}(-1)^{\frac{u}{2}+i}\binom{u}{\frac{u}{2}}\binom{u+1}{i}\left(x_{(u+2-2i)j+l}^{\Gamma,+}\right)\\
        &+\sum_{i=1}^{\frac{u+2}{2}}(-1)^{\frac{u}{2}+i}\binom{u}{\frac{u}{2}}\binom{u+1}{i-1}\left(x_{-(u+2-2i)j+l}^{\Gamma,+}\right)+\sum_{i=0}^{\frac{u}{2}}(-1)^{\frac{u}{2}+i}\binom{u}{\frac{u}{2}}\binom{u+1}{i}\left(x_{-(u+2-2i)j+l}^{\Gamma,+}\right)\\
        &+\sum_{i=1}^{\frac{u+2}{2}}(-1)^{\frac{u}{2}+i}\binom{u}{\frac{u}{2}}\binom{u+1}{i-1}\left(x_{(u+2-2i)j+l}^{\Gamma,+}\right)\Bigg)\\
        &=\sum_{k=1}^{\left\lfloor\frac{u-1}{2}\right\rfloor}\sum_{i=0}^{u+1}(-1)^{k+i}\binom{u}{k}\binom{u+1}{i}\left(x_{(u+2-2i)j+(u+1-2k)l}^{\Gamma,+}\right)+\sum_{i=0}^{u+1}(-1)^{i}\binom{u+1}{i}\left(x_{(u+2-2i)j+(u+1)l}^{\Gamma,+}\right)\\
        &+\sum_{k=1}^{\left\lfloor\frac{u-1}{2}\right\rfloor}\sum_{i=1}^{u+2}(-1)^{k+i}\binom{u}{k-1}\binom{u+1}{i-1}\left(x_{(u+2-2i)j+(u+1-2k)l}^{\Gamma,+}\right)\\
        &+\sum_{i=1}^{u+2}(-1)^{\left\lfloor\frac{u+1}{2}\right\rfloor+i}\binom{u}{\left\lfloor\frac{u-1}{2}\right\rfloor}\binom{u+1}{i-1}\left(x_{(u+2-2i)j+\left(u+1-2\left\lfloor\frac{u+1}{2}\right\rfloor\right)l}^{\Gamma,+}\right)\\
        &+\sum_{k=1}^{\left\lfloor\frac{u-1}{2}\right\rfloor}\sum_{i=0}^{u+1}(-1)^{k+i}\binom{u}{k-1}\binom{u+1}{i}\left(x_{(u+2-2i)j+(u+1-2k)l}^{\Gamma,+}\right)\\
        &+\sum_{i=0}^{u+1}(-1)^{\left\lfloor\frac{u+1}{2}\right\rfloor+i}\binom{u}{\left\lfloor\frac{u-1}{2}\right\rfloor}\binom{u+1}{i}\left(x_{(u+2-2i)j+\left(u+1-2\left\lfloor\frac{u+1}{2}\right\rfloor\right)l}^{\Gamma,+}\right)\\
        &+\sum_{k=1}^{\left\lfloor\frac{u-1}{2}\right\rfloor}\sum_{i=1}^{u+2}(-1)^{k+i}\binom{u}{k}\binom{u+1}{i-1}\left(x_{(u+2-2i)j+(u+1-2k)l}^{\Gamma,+}\right)+\sum_{i=1}^{u+2}(-1)^{i}\binom{u+1}{i-1}\left(x_{(u+2-2i)j+(u+1)l}^{\Gamma,+}\right)\\
        &+((u+1)\mod2)\Bigg(\sum_{i=1}^{\frac{u}{2}}(-1)^{\frac{u}{2}+i}\binom{u}{\frac{u}{2}}\binom{u+1}{i}\left(x_{(u+2-2i)j+l}^{\Gamma,+}\right)+(-1)^{\frac{u}{2}}\binom{u}{\frac{u}{2}}\left(x_{(u+2)j+l}^{\Gamma,+}\right)\\
        &+\sum_{i=1}^{\frac{u}{2}}(-1)^{\frac{u}{2}+i}\binom{u}{\frac{u}{2}}\binom{u+1}{i-1}\left(x_{-(u+2-2i)j+l}^{\Gamma,+}\right)+(-1)^{u+1}\binom{u}{\frac{u}{2}}\binom{u+1}{\frac{u}{2}}\left(x_{l}^{\Gamma,+}\right)\\
        &+\sum_{i=1}^{\frac{u}{2}}(-1)^{\frac{u}{2}+i}\binom{u}{\frac{u}{2}}\binom{u+1}{i}\left(x_{-(u+2-2i)j+l}^{\Gamma,+}\right)+(-1)^{\frac{u}{2}}\binom{u}{\frac{u}{2}}\left(x_{-(u+2)j+l}^{\Gamma,+}\right)\\
        &+\sum_{i=1}^{\frac{u}{2}}(-1)^{\frac{u}{2}+i}\binom{u}{\frac{u}{2}}\binom{u+1}{i-1}\left(x_{(u+2-2i)j+l}^{\Gamma,+}\right)+(-1)^{u+1}\binom{u}{\frac{u}{2}}\binom{u+1}{\frac{u}{2}}\left(x_{l}^{\Gamma,+}\right)\Bigg)\\
        &=\sum_{k=1}^{\left\lfloor\frac{u-1}{2}\right\rfloor}\sum_{i=1}^{u+1}(-1)^{k+i}\binom{u}{k}\binom{u+1}{i}\left(x_{(u+2-2i)j+(u+1-2k)l}^{\Gamma,+}\right)+\sum_{k=1}^{\left\lfloor\frac{u-1}{2}\right\rfloor}(-1)^{k}\binom{u}{k}\left(x_{(u+2)j+(u+1-2k)l}^{\Gamma,+}\right)\\
        &+\left(x_{(u+2)j+(u+1)l}^{\Gamma,+}\right)+\sum_{i=1}^{u+1}(-1)^{i}\binom{u+1}{i}\left(x_{(u+2-2i)j+(u+1)l}^{\Gamma,+}\right)\\
        &+\sum_{k=1}^{\left\lfloor\frac{u-1}{2}\right\rfloor}\sum_{i=1}^{u+1}(-1)^{k+i}\binom{u}{k-1}\binom{u+1}{i-1}\left(x_{(u+2-2i)j+(u+1-2k)l}^{\Gamma,+}\right)\\
        &+\sum_{k=1}^{\left\lfloor\frac{u-1}{2}\right\rfloor}(-1)^{k+u+2}\binom{u}{k-1}\left(x_{-(u+2)j+(u+1-2k)l}^{\Gamma,+}\right)\\
        &+\sum_{i=1}^{u+1}(-1)^{\left\lfloor\frac{u+1}{2}\right\rfloor+i}\binom{u}{\left\lfloor\frac{u-1}{2}\right\rfloor}\binom{u+1}{i-1}\left(x_{(u+2-2i)j+\left(u+1-2\left\lfloor\frac{u+1}{2}\right\rfloor\right)l}^{\Gamma,+}\right)\\
        &+(-1)^{\left\lfloor\frac{3u+5}{2}\right\rfloor}\binom{u}{\left\lfloor\frac{u-1}{2}\right\rfloor}\left(x_{-(u+2)j+\left(u+1-2\left\lfloor\frac{u+1}{2}\right\rfloor\right)l}^{\Gamma,+}\right)\\
        &+\sum_{k=1}^{\left\lfloor\frac{u-1}{2}\right\rfloor}\sum_{i=1}^{u+1}(-1)^{k+i}\binom{u}{k-1}\binom{u+1}{i}\left(x_{(u+2-2i)j+(u+1-2k)l}^{\Gamma,+}\right)\\
        &+\sum_{k=1}^{\left\lfloor\frac{u-1}{2}\right\rfloor}(-1)^{k}\binom{u}{k-1}\left(x_{(u+2)j+(u+1-2k)l}^{\Gamma,+}\right)\\
        &+(-1)^{\left\lfloor\frac{u+1}{2}\right\rfloor}\binom{u}{\left\lfloor\frac{u-1}{2}\right\rfloor}\left(x_{(u+2)j+\left(u+1-2\left\lfloor\frac{u+1}{2}\right\rfloor\right)l}^{\Gamma,+}\right)\\
        &+\sum_{i=1}^{u+1}(-1)^{\left\lfloor\frac{u+1}{2}\right\rfloor+i}\binom{u}{\left\lfloor\frac{u-1}{2}\right\rfloor}\binom{u+1}{i}\left(x_{(u+2-2i)j+\left(u+1-2\left\lfloor\frac{u+1}{2}\right\rfloor\right)l}^{\Gamma,+}\right)\\
        &+\sum_{k=1}^{\left\lfloor\frac{u-1}{2}\right\rfloor}\sum_{i=1}^{u+1}(-1)^{k+i}\binom{u}{k}\binom{u+1}{i-1}\left(x_{(u+2-2i)j+(u+1-2k)l}^{\Gamma,+}\right)\\
        &+\sum_{k=1}^{\left\lfloor\frac{u-1}{2}\right\rfloor}(-1)^{k+u+2}\binom{u}{k}\left(x_{-(u+2)j+(u+1-2k)l}^{\Gamma,+}\right)\\
        &+\sum_{i=1}^{u+1}(-1)^{i}\binom{u+1}{i-1}\left(x_{(u+2-2i)j+(u+1)l}^{\Gamma,+}\right)+(-1)^{u+2}\left(x_{-(u+2)j+(u+1)l}^{\Gamma,+}\right)\\
        &+((u+1)\mod2)\Bigg(\sum_{i=1}^{\frac{u}{2}}(-1)^{\frac{u}{2}+i}\binom{u}{\frac{u}{2}}\left(\binom{u+1}{i}+\binom{u+1}{i-1}\right)\left(x_{(u+2-2i)j+l}^{\Gamma,+}\right)\\
        &+(-1)^{\frac{u}{2}}\binom{u}{\frac{u}{2}}\left(x_{(u+2)j+l}^{\Gamma,+}\right)\\
        &+\sum_{i=1}^{\frac{u}{2}}(-1)^{\frac{u}{2}+i}\binom{u}{\frac{u}{2}}\left(\binom{u+1}{i-1}+\binom{u+1}{i}\right)\left(x_{-(u+2-2i)j+l}^{\Gamma,+}\right)+(-1)^{u+1}\binom{u}{\frac{u}{2}}\binom{u+1}{\frac{u}{2}}\left(x_{l}^{\Gamma,+}\right)\\
        &+(-1)^{\frac{u}{2}}\binom{u}{\frac{u}{2}}\left(x_{-(u+2)j+l}^{\Gamma,+}\right)+(-1)^{u+1}\binom{u}{\frac{u}{2}}\binom{u+1}{\frac{u}{2}}\left(x_{l}^{\Gamma,+}\right)\Bigg)\\
        &=\sum_{k=1}^{\left\lfloor\frac{u-1}{2}\right\rfloor}\sum_{i=1}^{u+1}(-1)^{k+i}\Bigg(\binom{u}{k}\bigg(\binom{u+1}{i}+\binom{u+1}{i-1}\bigg)\\
        &+\binom{u}{k-1}\bigg(\binom{u+1}{i}+\binom{u+1}{i-1}\bigg)\Bigg)\left(x_{(u+2-2i)j+(u+1-2k)l}^{\Gamma,+}\right)\\
        &+\sum_{k=1}^{\left\lfloor\frac{u-1}{2}\right\rfloor}(-1)^{k}\left(\binom{u}{k}+\binom{u}{k-1}\right)\left(x_{(u+2)j+(u+1-2k)l}^{\Gamma,+}\right)+\left(x_{(u+2)j+(u+1)l}^{\Gamma,+}\right)\\
        &+\sum_{i=1}^{u+1}(-1)^{i}\left(\binom{u+1}{i}+\binom{u+1}{i-1}\right)\left(x_{(u+2-2i)j+(u+1)l}^{\Gamma,+}\right)\\
        &+\sum_{k=1}^{\left\lfloor\frac{u-1}{2}\right\rfloor}(-1)^{k+u+2}\left(\binom{u}{k-1}+\binom{u}{k}\right)\left(x_{-(u+2)j+(u+1-2k)l}^{\Gamma,+}\right)\\
        &+\sum_{i=1}^{u+1}(-1)^{\left\lfloor\frac{u+1}{2}\right\rfloor+i}\binom{u}{\left\lfloor\frac{u-1}{2}\right\rfloor}\left(\binom{u+1}{i-1}+\binom{u+1}{i}\right)\left(x_{(u+2-2i)j+\left(u+1-2\left\lfloor\frac{u+1}{2}\right\rfloor\right)l}^{\Gamma,+}\right)\\
        &+(-1)^{\left\lfloor\frac{3u+5}{2}\right\rfloor}\binom{u}{\left\lfloor\frac{u-1}{2}\right\rfloor}\left(x_{-(u+2)j+\left(u+1-2\left\lfloor\frac{u+1}{2}\right\rfloor\right)l}^{\Gamma,+}\right)+(-1)^{\left\lfloor\frac{u+1}{2}\right\rfloor}\binom{u}{\left\lfloor\frac{u-1}{2}\right\rfloor}\left(x_{(u+2)j+\left(u+1-2\left\lfloor\frac{u+1}{2}\right\rfloor\right)l}^{\Gamma,+}\right)\\
        &+(-1)^{u+2}\left(x_{-(u+2)j+(u+1)l}^{\Gamma,+}\right)+((u+1)\mod2)\Bigg(\sum_{i=1}^{\frac{u}{2}}(-1)^{\frac{u}{2}+i}\binom{u}{\frac{u}{2}}\binom{u+2}{i}\left(x_{(u+2-2i)j+l}^{\Gamma,+}\right)\\
        &+(-1)^{\frac{u}{2}}\binom{u}{\frac{u}{2}}\left(x_{(u+2)j+l}^{\Gamma,+}\right)+\sum_{i=1}^{\frac{u}{2}}(-1)^{\frac{u}{2}+i}\binom{u}{\frac{u}{2}}\binom{u+2}{i}\left(x_{-(u+2-2i)j+l}^{\Gamma,+}\right)\\
        &+(-1)^{u+1}\binom{u}{\frac{u}{2}}\binom{u+1}{\frac{u}{2}}\left(x_{l}^{\Gamma,+}\right)\\
        &+(-1)^{\frac{u}{2}}\binom{u}{\frac{u}{2}}\left(x_{-(u+2)j+l}^{\Gamma,+}\right)+(-1)^{u+1}\binom{u}{\frac{u}{2}}\binom{u+1}{\frac{u}{2}}\left(x_{l}^{\Gamma,+}\right)\Bigg)\textnormal{ by Pascal's Rule}\\
        &=\sum_{k=1}^{\left\lfloor\frac{u-1}{2}\right\rfloor}\sum_{i=1}^{u+1}(-1)^{k+i}\binom{u+1}{k}\binom{u+2}{i}\left(x_{(u+2-2i)j+(u+1-2k)l}^{\Gamma,+}\right)\\
        &+\sum_{k=1}^{\left\lfloor\frac{u-1}{2}\right\rfloor}(-1)^{k}\binom{u+1}{k}\left(x_{(u+2)j+(u+1-2k)l}^{\Gamma,+}\right)\\
        &+\left(x_{(u+2)j+(u+1)l}^{\Gamma,+}\right)+\sum_{i=1}^{u+1}(-1)^{i}\binom{u+2}{i}\left(x_{(u+2-2i)j+(u+1)l}^{\Gamma,+}\right)\\
        &+\sum_{k=1}^{\left\lfloor\frac{u-1}{2}\right\rfloor}(-1)^{k+u+2}\binom{u+1}{k}\left(x_{-(u+2)j+(u+1-2k)l}^{\Gamma,+}\right)\\
        &+\sum_{i=1}^{u+1}(-1)^{\left\lfloor\frac{u+1}{2}\right\rfloor+i}\binom{u}{\left\lfloor\frac{u-1}{2}\right\rfloor}\binom{u+2}{i}\left(x_{(u+2-2i)j+\left(u+1-2\left\lfloor\frac{u+1}{2}\right\rfloor\right)l}^{\Gamma,+}\right)+(-1)^{u+2}\left(x_{-(u+2)j+(u+1)l}^{\Gamma,+}\right)\\
        &+(-1)^{\left\lfloor\frac{3u+5}{2}\right\rfloor}\binom{u}{\left\lfloor\frac{u-1}{2}\right\rfloor}\left(x_{-(u+2)j+\left(u+1-2\left\lfloor\frac{u+1}{2}\right\rfloor\right)l}^{\Gamma,+}\right)+(-1)^{\left\lfloor\frac{u+1}{2}\right\rfloor}\binom{u}{\left\lfloor\frac{u-1}{2}\right\rfloor}\left(x_{(u+2)j+\left(u+1-2\left\lfloor\frac{u+1}{2}\right\rfloor\right)l}^{\Gamma,+}\right)\\
        &+((u+1)\mod2)\Bigg(\sum_{i=0}^{\frac{u}{2}}(-1)^{\frac{u}{2}+i}\binom{u}{\frac{u}{2}}\binom{u+2}{i}\left(x_{(u+2-2i)j+l}^{\Gamma,+}\right)\\
        &+\sum_{i=0}^{\frac{u}{2}}(-1)^{\frac{u}{2}+i}\binom{u}{\frac{u}{2}}\binom{u+2}{i}\left(x_{-(u+2-2i)j+l}^{\Gamma,+}\right)\\
        &+2(-1)^{u+1}\binom{u}{\frac{u}{2}}\binom{u+1}{\frac{u}{2}}\left(x_{l}^{\Gamma,+}\right)\Bigg)\textnormal{ by Pascal's Rule}\\
        &=\sum_{k=1}^{\left\lfloor\frac{u-1}{2}\right\rfloor}\sum_{i=0}^{u+1}(-1)^{k+i}\binom{u+1}{k}\binom{u+2}{i}\left(x_{(u+2-2i)j+(u+1-2k)l}^{\Gamma,+}\right)\\
        &+\sum_{i=0}^{u+1}(-1)^{i}\binom{u+2}{i}\left(x_{(u+2-2i)j+(u+1)l}^{\Gamma,+}\right)\\
        &+\sum_{k=0}^{\left\lfloor\frac{u-1}{2}\right\rfloor}(-1)^{k+u+2}\binom{u+1}{k}\left(x_{-(u+2)j+(u+1-2k)l}^{\Gamma,+}\right)\\
        &+\sum_{i=0}^{u+2}(-1)^{\left\lfloor\frac{u+1}{2}\right\rfloor+i}\binom{u}{\left\lfloor\frac{u-1}{2}\right\rfloor}\binom{u+2}{i}\left(x_{(u+2-2i)j+\left(u+1-2\left\lfloor\frac{u+1}{2}\right\rfloor\right)l}^{\Gamma,+}\right)\\
        &+((u+1)\mod2)\Bigg(\sum_{i=0}^{\frac{u}{2}}(-1)^{\frac{u}{2}+i}\binom{u}{\frac{u}{2}}\binom{u+2}{i}\left(x_{(u+2-2i)j+l}^{\Gamma,+}\right)\\
        &+\sum_{i=\frac{u+4}{2}}^{u+2}(-1)^{\frac{3u+4}{2}-i}\binom{u}{\frac{u}{2}}\binom{u+2}{i}\left(x_{(u+2-2i)j+l}^{\Gamma,+}\right)+(-1)^{u+1}\binom{u}{\frac{u}{2}}\binom{u+2}{\frac{u+2}{2}}\left(x_{l}^{\Gamma,+}\right)\Bigg)\\
        &=\sum_{k=0}^{\left\lfloor\frac{u-1}{2}\right\rfloor}\sum_{i=0}^{u+2}(-1)^{k+i}\binom{u+1}{k}\binom{u+2}{i}\left(x_{(u+2-2i)j+(u+1-2k)l}^{\Gamma,+}\right)\\
        &+(u\mod2)\sum_{i=0}^{u+2}(-1)^{\frac{u+1}{2}+i}\binom{u}{\frac{u-1}{2}}\binom{u+2}{i}\left(x_{(u+2-2i)j}^{\Gamma,+}\right)\\
        &+((u+1)\mod2)\Bigg(\sum_{i=0}^{u+2}(-1)^{\frac{u}{2}+i}\binom{u}{\frac{u-2}{2}}\binom{u+2}{i}\left(x_{(u+2-2i)j+l}^{\Gamma,+}\right)\\
        &+\sum_{i=0}^{\frac{u+2}{2}}(-1)^{\frac{u}{2}+i}\binom{u}{\frac{u}{2}}\binom{u+2}{i}\left(x_{(u+2-2i)j+l}^{\Gamma,+}\right)\\
        &+\sum_{i=\frac{u+4}{2}}^{u+2}(-1)^{\frac{u}{2}+i}\binom{u}{\frac{u}{2}}\binom{u+2}{i}\left(x_{(u+2-2i)j+l}^{\Gamma,+}\right)\Bigg)\\
        &=\sum_{k=0}^{\left\lfloor\frac{u-1}{2}\right\rfloor}\sum_{i=0}^{u+2}(-1)^{k+i}\binom{u+1}{k}\binom{u+2}{i}\left(x_{(u+2-2i)j+(u+1-2k)l}^{\Gamma,+}\right)\\
        &+(u\mod2)\Bigg(\sum_{i=0}^{\frac{u+1}{2}}(-1)^{\frac{u+1}{2}+i}\binom{u}{\frac{u-1}{2}}\binom{u+2}{i}\left(x_{(u+2-2i)j}^{\Gamma,+}\right)\\
        &+\sum_{i=\frac{u+3}{2}}^{u+2}(-1)^{\frac{u+1}{2}+i}\binom{u}{\frac{u-1}{2}}\binom{u+2}{i}\left(x_{(u+2-2i)j}^{\Gamma,+}\right)\Bigg)\\
        &+((u+1)\mod2)\Bigg(\sum_{i=0}^{u+2}(-1)^{\frac{u}{2}+i}\binom{u}{\frac{u-2}{2}}\binom{u+2}{i}\left(x_{(u+2-2i)j+l}^{\Gamma,+}\right)\\
        &+\sum_{i=0}^{u+2}(-1)^{\frac{u}{2}+i}\binom{u}{\frac{u}{2}}\binom{u+2}{i}\left(x_{(u+2-2i)j+l}^{\Gamma,+}\right)\Bigg)\\
        &=\sum_{k=0}^{\left\lfloor\frac{u-1}{2}\right\rfloor}\sum_{i=0}^{u+2}(-1)^{k+i}\binom{u+1}{k}\binom{u+2}{i}\left(x_{(u+2-2i)j+(u+1-2k)l}^{\Gamma,+}\right)\\
        &+(u\mod2)\Bigg(\sum_{i=0}^{\frac{u+1}{2}}(-1)^{\frac{u+1}{2}+i}\binom{u}{\frac{u-1}{2}}\binom{u+2}{i}\left(x_{(u+2-2i)j}^{\Gamma,+}\right)\\
        &+\sum_{i=\frac{u+3}{2}}^{u+2}(-1)^{\frac{u+3}{2}+i}\binom{u}{\frac{u-1}{2}}\binom{u+2}{i}\left(x_{-(u+2-2i)j}^{\Gamma,+}\right)\Bigg)\\
        &+((u+1)\mod2)\sum_{i=0}^{u+2}(-1)^{\frac{u}{2}+i}\left(\binom{u}{\frac{u-2}{2}}+\binom{u}{\frac{u}{2}}\right)\binom{u+2}{i}\left(x_{(u+2-2i)j+l}^{\Gamma,+}\right)\\
        &=\sum_{k=0}^{\left\lfloor\frac{u-1}{2}\right\rfloor}\sum_{i=0}^{u+2}(-1)^{k+i}\binom{u+1}{k}\binom{u+2}{i}\left(x_{(u+2-2i)j+(u+1-2k)l}^{\Gamma,+}\right)\\
        &+(u\mod2)\Bigg(\sum_{i=0}^{\frac{u+1}{2}}(-1)^{\frac{u+1}{2}+i}\binom{u}{\frac{u-1}{2}}\binom{u+2}{i}\left(x_{(u+2-2i)j}^{\Gamma,+}\right)\\
        &+\sum_{i=0}^{\frac{u+1}{2}}(-1)^{\frac{3u+7}{2}-i}\binom{u}{\frac{u-1}{2}}\binom{u+2}{i}\left(x_{(u+2-2i)j}^{\Gamma,+}\right)\Bigg)\\
        &+((u+1)\mod2)\sum_{i=0}^{u+2}(-1)^{\frac{u}{2}+i}\binom{u+1}{\frac{u}{2}}\binom{u+2}{i}\left(x_{(u+2-2i)j+l}^{\Gamma,+}\right)\textnormal{ by Pascal's Rule}\\
        &=\sum_{k=0}^{\left\lfloor\frac{u}{2}\right\rfloor}\sum_{i=0}^{u+2}(-1)^{k+i}\binom{u+1}{k}\binom{u+2}{i}\left(x_{(u+2-2i)j+(u+1-2k)l}^{\Gamma,+}\right)\\
        &+(u\mod2)\Bigg(\sum_{i=0}^{\frac{u+1}{2}}(-1)^{\frac{u+1}{2}+i}\binom{u}{\frac{u-1}{2}}\binom{u+2}{i}\left(x_{(u+2-2i)j}^{\Gamma,+}\right)\\
        &+\sum_{i=0}^{\frac{u+1}{2}}(-1)^{\frac{u+1}{2}+i}\binom{u}{\frac{u-1}{2}}\binom{u+2}{i}\left(x_{(u+2-2i)j}^{\Gamma,+}\right)\Bigg)\\
        &=\sum_{k=0}^{\left\lfloor\frac{u}{2}\right\rfloor}\sum_{i=0}^{u+2}(-1)^{k+i}\binom{u+1}{k}\binom{u+2}{i}\left(x_{(u+2-2i)j+(u+1-2k)l}^{\Gamma,+}\right)\\
        &+2(u\mod2)\sum_{i=0}^{\frac{u+1}{2}}(-1)^{\frac{u+1}{2}+i}\binom{u}{\frac{u-1}{2}}\binom{u+2}{i}\left(x_{(u+2-2i)j}^{\Gamma,+}\right)\\
        &=\sum_{k=0}^{\left\lfloor\frac{u}{2}\right\rfloor}\sum_{i=0}^{u+2}(-1)^{k+i}\binom{u+1}{k}\binom{u+2}{i}\left(x_{(u+2-2i)j+(u+1-2k)l}^{\Gamma,+}\right)\\
        &+(u\mod2)\sum_{i=0}^{\frac{u+1}{2}}(-1)^{\frac{u+1}{2}+i}\binom{u+1}{\frac{u+1}{2}}\binom{u+2}{i}\left(x_{(u+2-2i)j}^{\Gamma,+}\right)
    \end{align*}
    \hfill\qedsymbol

 \subsection*{Proof of Proposition \ref{p_u}}
 
 Recall that the statement we wish to prove is for all $u\in\mathbb N$
    \begin{align*}
        p^{\Gamma}_u(j,l)
        &=\sum_{k=0}^{\left\lfloor\frac{u-1}{2}\right\rfloor}\sum_{i=0}^{u}(-1)^{k+i}\binom{u}{k}\binom{u}{i}\left(h_{(u-2i)j+(u-2k)l}^{\Gamma}\right)\nonumber\\
        &+((u+1)\mod2)\sum_{i=0}^{u}(-1)^{\frac{u}{2}+i}\binom{u-1}{\frac{u-2}{2}}\binom{u}{i}\left(h_{(u-2i)j}^{\Gamma}\right).
    \end{align*}
    
    By the definition of $p_u^\Gamma(j,l)$ we have:
    \begin{align*}
        p^{\Gamma}_u(j,l)&=\left[x_{j}^{\Gamma,+},D^{\Gamma,-}_{u-1,1}(j,l)\right]\\
        &=\sum_{k=0}^{\left\lfloor\frac{u-2}{2}\right\rfloor}\sum_{i=0}^{u}(-1)^{k+i}\binom{u-1}{k}\binom{u}{i}\left[x_{j}^{\Gamma,+},\left(x_{(u-2i)l+(u-1-2k)j}^{\Gamma,-}\right)\right]\\
        &+(u\mod2)\sum_{i=0}^{\frac{u-1}{2}}(-1)^{\frac{u-1}{2}+i}\binom{u-1}{\frac{u-1}{2}}\binom{u}{i}\left[x_{j}^{\Gamma,+},\left(x_{(u-2i)l}^{\Gamma,-}\right)\right]\textnormal{ by Proposition \ref{Du1}\eqref{D-u1}}\\
        &=\sum_{k=0}^{\left\lfloor\frac{u-2}{2}\right\rfloor}\sum_{i=0}^{u}(-1)^{k+i}\binom{u-1}{k}\binom{u}{i}\left(\left(h_{(u-2i)l+(u-2k)j}^{\Gamma}\right)-\left(h_{(u-2i)l+(u-2-2k)j}^{\Gamma}\right)\right)\\
        &+(u\mod2)\sum_{i=0}^{\frac{u-1}{2}}(-1)^{\frac{u-1}{2}+i}\binom{u-1}{\frac{u-1}{2}}\binom{u}{i}\left(\left(h_{j+(u-2i)l}^{\Gamma}\right)-\left(h_{-j+(u-2i)l}^{\Gamma}\right)\right)\\
        &=\sum_{i=0}^{\left\lfloor\frac{u-2}{2}\right\rfloor}\sum_{k=0}^{u}(-1)^{k+i}\binom{u}{k}\binom{u-1}{i}\left(\left(h_{(u-2i)j+(u-2k)l}^{\Gamma}\right)-\left(h_{(u-2-2i)j+(u-2k)l}^{\Gamma}\right)\right)\\
        &+(u\mod2)\sum_{k=0}^{\frac{u-1}{2}}(-1)^{\frac{u-1}{2}+k}\binom{u-1}{\frac{u-1}{2}}\binom{u}{k}\left(\left(h_{j+(u-2k)l}^{\Gamma}\right)-\left(h_{-j+(u-2k)l}^{\Gamma}\right)\right)\\
        &=\sum_{i=0}^{\left\lfloor\frac{u-2}{2}\right\rfloor}\sum_{k=0}^{\left\lfloor\frac{u}{2}\right\rfloor}(-1)^{k+i}\binom{u}{k}\binom{u-1}{i}\left(\left(h_{(u-2i)j+(u-2k)l}^{\Gamma}\right)-\left(h_{(u-2-2i)j+(u-2k)l}^{\Gamma}\right)\right)\\
        &+\sum_{i=0}^{\left\lfloor\frac{u-2}{2}\right\rfloor}\sum_{k=\left\lfloor\frac{u+2}{2}\right\rfloor}^{u}(-1)^{k+i}\binom{u}{k}\binom{u-1}{i}\left(\left(h_{(u-2i)j+(u-2k)l}^{\Gamma}\right)-\left(h_{(u-2-2i)j+(u-2k)l}^{\Gamma}\right)\right)\\
        &+(u\mod2)\sum_{k=0}^{\frac{u-1}{2}}(-1)^{\frac{u-1}{2}+k}\binom{u-1}{\frac{u-1}{2}}\binom{u}{k}\left(h_{j+(u-2k)l}^{\Gamma}\right)\\
        &+(u\mod2)\sum_{k=0}^{\frac{u-1}{2}}(-1)^{\frac{u+1}{2}+k}\binom{u-1}{\frac{u-1}{2}}\binom{u}{k}\left(h_{j-(u-2k)l}^{\Gamma}\right)\\
        &=\sum_{i=0}^{\left\lfloor\frac{u-2}{2}\right\rfloor}\sum_{k=0}^{\left\lfloor\frac{u}{2}\right\rfloor}(-1)^{k+i}\binom{u}{k}\binom{u-1}{i}\left(h_{(u-2i)j+(u-2k)l}^{\Gamma}\right)\\
        &+\sum_{i=0}^{\left\lfloor\frac{u-2}{2}\right\rfloor}\sum_{k=0}^{\left\lfloor\frac{u}{2}\right\rfloor}(-1)^{k+i+1}\binom{u}{k}\binom{u-1}{i}\left(h_{(u-2-2i)j+(u-2k)l}^{\Gamma}\right)\\
        &+\sum_{i=0}^{\left\lfloor\frac{u-2}{2}\right\rfloor}\sum_{k=\left\lfloor\frac{u+2}{2}\right\rfloor}^{u}(-1)^{k+i}\binom{u}{k}\binom{u-1}{i}\left(h_{(u-2i)j+(u-2k)l}^{\Gamma}\right)\\
        &+\sum_{i=0}^{\left\lfloor\frac{u-2}{2}\right\rfloor}\sum_{k=\left\lfloor\frac{u+2}{2}\right\rfloor}^{u}(-1)^{k+i+1}\binom{u}{k}\binom{u-1}{i}\left(h_{(u-2-2i)j+(u-2k)l}^{\Gamma}\right)\\
        &+(u\mod2)\sum_{k=0}^{\frac{u-1}{2}}(-1)^{\frac{u-1}{2}+k}\binom{u-1}{\frac{u-1}{2}}\binom{u}{k}\left(h_{j+(u-2k)l}^{\Gamma}\right)\\
        &+(u\mod2)\sum_{k=0}^{\frac{u-1}{2}}(-1)^{\frac{u+1}{2}+k}\binom{u-1}{\frac{u-1}{2}}\binom{u}{k}\left(h_{-j+(u-2k)l}^{\Gamma}\right)\\
        &=\sum_{i=0}^{\left\lfloor\frac{u-2}{2}\right\rfloor}\sum_{k=0}^{\left\lfloor\frac{u}{2}\right\rfloor}(-1)^{k+i}\binom{u}{k}\binom{u-1}{i}\left(h_{(u-2i)j+(u-2k)l}^{\Gamma}\right)+\sum_{i=1}^{\left\lfloor\frac{u}{2}\right\rfloor}\sum_{k=0}^{\left\lfloor\frac{u}{2}\right\rfloor}(-1)^{k+i}\binom{u}{k}\binom{u-1}{i-1}\left(h_{(u-2i)j+(u-2k)l}^{\Gamma}\right)\\
        &+\sum_{i=0}^{\left\lfloor\frac{u-2}{2}\right\rfloor}\sum_{k=\left\lfloor\frac{u+2}{2}\right\rfloor}^{u}(-1)^{k+i}\binom{u}{k}\binom{u-1}{i}\left(h_{(u-2i)j+(u-2k)l}^{\Gamma}\right)\\
        &+\sum_{i=1}^{\left\lfloor\frac{u}{2}\right\rfloor}\sum_{k=\left\lfloor\frac{u+2}{2}\right\rfloor}^{u}(-1)^{k+i}\binom{u}{k}\binom{u-1}{i-1}\left(h_{(u-2i)j+(u-2k)l}^{\Gamma}\right)\\
        &+(u\mod2)\sum_{k=0}^{\frac{u-1}{2}}(-1)^{\frac{u-1}{2}+k}\binom{u-1}{\frac{u-1}{2}}\binom{u}{k}\left(h_{j+(u-2k)l}^{\Gamma}\right)\\
        &+(u\mod2)\sum_{k=0}^{\frac{u-1}{2}}(-1)^{\frac{u+1}{2}+k}\binom{u-1}{\frac{u-1}{2}}\binom{u}{k}\left(h_{-j+(u-2k)l}^{\Gamma}\right)\\
        &=\sum_{k=0}^{\left\lfloor\frac{u}{2}\right\rfloor}(-1)^{k}\binom{u}{k}\left(h_{uj+(u-2k)l}^{\Gamma}\right)+\sum_{i=1}^{\left\lfloor\frac{u-2}{2}\right\rfloor}\sum_{k=0}^{\left\lfloor\frac{u}{2}\right\rfloor}(-1)^{k+i}\binom{u}{k}\binom{u-1}{i}\left(h_{(u-2i)j+(u-2k)l}^{\Gamma}\right)\\
        &+\sum_{i=1}^{\left\lfloor\frac{u-2}{2}\right\rfloor}\sum_{k=0}^{\left\lfloor\frac{u}{2}\right\rfloor}(-1)^{k+i}\binom{u}{k}\binom{u-1}{i-1}\left(h_{(u-2i)j+(u-2k)l}^{\Gamma}\right)+\sum_{k=0}^{\left\lfloor\frac{u}{2}\right\rfloor}(-1)^{k+\left\lfloor\frac{u}{2}\right\rfloor}\binom{u}{k}\binom{u-1}{\left\lfloor\frac{u-2}{2}\right\rfloor}\left(h_{(u-2\left\lfloor\frac{u}{2}\right\rfloor)j+(u-2k)l}^{\Gamma}\right)\\
        &+\sum_{i=0}^{\left\lfloor\frac{u-2}{2}\right\rfloor}\sum_{k=\left\lfloor\frac{u+2}{2}\right\rfloor}^{u}(-1)^{k+i}\binom{u}{k}\binom{u-1}{i}\left(h_{-(u-2i)j-(u-2k)l}^{\Gamma}\right)\\
        &+\sum_{i=1}^{\left\lfloor\frac{u}{2}\right\rfloor}\sum_{k=\left\lfloor\frac{u+2}{2}\right\rfloor}^{u}(-1)^{k+i}\binom{u}{k}\binom{u-1}{i-1}\left(h_{-(u-2i)j-(u-2k)l}^{\Gamma}\right)\\
        &+(u\mod2)\sum_{k=0}^{\frac{u-1}{2}}(-1)^{\frac{u-1}{2}+k}\binom{u-1}{\frac{u-1}{2}}\binom{u}{k}\left(h_{j+(u-2k)l}^{\Gamma}\right)\\
        &+(u\mod2)\sum_{k=0}^{\frac{u-1}{2}}(-1)^{\frac{u+1}{2}+k}\binom{u-1}{\frac{u-1}{2}}\binom{u}{k}\left(h_{-j+(u-2k)l}^{\Gamma}\right)\\
        &=\sum_{k=0}^{\left\lfloor\frac{u}{2}\right\rfloor}(-1)^{k}\binom{u}{k}\left(h_{uj+(u-2k)l}^{\Gamma}\right)+\sum_{i=1}^{\left\lfloor\frac{u-2}{2}\right\rfloor}\sum_{k=0}^{\left\lfloor\frac{u}{2}\right\rfloor}(-1)^{k+i}\binom{u}{k}\left(\binom{u-1}{i}+\binom{u-1}{i-1}\right)\left(h_{(u-2i)j+(u-2k)l}^{\Gamma}\right)\\
        &+((u+1)\mod2)\sum_{k=0}^{\frac{u}{2}}(-1)^{k+\frac{u}{2}}\binom{u}{k}\binom{u-1}{\frac{u-2}{2}}\left(h_{(u-2k)l}^{\Gamma}\right)\\
        &+\sum_{i=\left\lfloor\frac{u+3}{2}\right\rfloor}^{u}\sum_{k=0}^{\left\lfloor\frac{u-1}{2}\right\rfloor}(-1)^{2u-k-i}\binom{u}{k}\binom{u-1}{i-1}\left(h_{(u-2i)j+(u-2k)l}^{\Gamma}\right)\\
        &+\sum_{i=\left\lfloor\frac{u+1}{2}\right\rfloor}^{u-1}\sum_{k=0}^{\left\lfloor\frac{u-1}{2}\right\rfloor}(-1)^{2u-k-i}\binom{u}{k}\binom{u-1}{i}\left(h_{(u-2i)j+(u-2k)l}^{\Gamma}\right)\\
        &+(u\mod2)\sum_{k=0}^{\frac{u-1}{2}}(-1)^{\frac{u-1}{2}+k}\binom{u}{\frac{u-1}{2}}\binom{u}{k}\left(h_{j+(u-2k)l}^{\Gamma}\right)\\
        &+(u\mod2)\sum_{k=0}^{\frac{u-1}{2}}(-1)^{\frac{u+1}{2}+k}\binom{u-1}{\frac{u-1}{2}}\binom{u}{k}\left(h_{-j+(u-2k)l}^{\Gamma}\right)\\
        &=\sum_{k=0}^{\left\lfloor\frac{u}{2}\right\rfloor}(-1)^{k}\binom{u}{k}\left(h_{uj+(u-2k)l}^{\Gamma}\right)+\sum_{i=1}^{\left\lfloor\frac{u-2}{2}\right\rfloor}\sum_{k=0}^{\left\lfloor\frac{u}{2}\right\rfloor}(-1)^{k+i}\binom{u}{k}\binom{u}{i}\left(h_{(u-2i)j+(u-2k)l}^{\Gamma}\right)\\
        &+((u+1)\mod2)\sum_{k=0}^{\frac{u}{2}}(-1)^{k+\frac{u}{2}}\binom{u}{k}\binom{u-1}{\frac{u-2}{2}}\left(h_{(u-2k)l}^{\Gamma}\right)+\sum_{k=0}^{\left\lfloor\frac{u-1}{2}\right\rfloor}(-1)^{u+k}\binom{u}{k}\left(h_{-uj+(u-2k)l}^{\Gamma}\right)\\
        &+\sum_{i=\left\lfloor\frac{u+3}{2}\right\rfloor}^{u-1}\sum_{k=0}^{\left\lfloor\frac{u-1}{2}\right\rfloor}(-1)^{k+i}\binom{u}{k}\binom{u-1}{i-1}\left(h_{(u-2i)j+(u-2k)l}^{\Gamma}\right)\\
        &+\sum_{i=\left\lfloor\frac{u+3}{2}\right\rfloor}^{u-1}\sum_{k=0}^{\left\lfloor\frac{u-1}{2}\right\rfloor}(-1)^{k+i}\binom{u}{k}\binom{u-1}{i}\left(h_{(u-2i)j+(u-2k)l}^{\Gamma}\right)\\
        &+\sum_{k=0}^{\left\lfloor\frac{u-1}{2}\right\rfloor}(-1)^{k+\left\lfloor\frac{u+1}{2}\right\rfloor}\binom{u}{k}\binom{u-1}{\left\lfloor\frac{u+1}{2}\right\rfloor}\left(h_{(u-2\left\lfloor\frac{u+1}{2}\right\rfloor)j+(u-2k)l}^{\Gamma}\right)\\
        &+(u\mod2)\sum_{k=0}^{\frac{u-1}{2}}(-1)^{\frac{u-1}{2}+k}\binom{u}{\frac{u-1}{2}}\binom{u}{k}\left(h_{j+(u-2k)l}^{\Gamma}\right)\\
        &+(u\mod2)\sum_{k=0}^{\frac{u-1}{2}}(-1)^{\frac{u+1}{2}+k}\binom{u-1}{\frac{u-1}{2}}\binom{u}{k}\left(h_{-j+(u-2k)l}^{\Gamma}\right)\textnormal{ by Pascal's Rule}\\
        &=\sum_{i=0}^{\left\lfloor\frac{u-2}{2}\right\rfloor}\sum_{k=0}^{\left\lfloor\frac{u}{2}\right\rfloor}(-1)^{k+i}\binom{u}{k}\binom{u}{i}\left(h_{(u-2i)j+(u-2k)l}^{\Gamma}\right)\\
        &+((u+1)\mod2)\sum_{k=0}^{\frac{u-2}{2}}(-1)^{k+\frac{u}{2}}\binom{u}{k}\binom{u}{\frac{u}{2}}\left(h_{(u-2k)l}^{\Gamma}\right)\\
        &+((u+1)\mod2)(-1)^{u}\binom{u}{\frac{u}{2}}\binom{u-1}{\frac{u-2}{2}}\left(h_{0}^{\Gamma}\right)+\sum_{k=0}^{\left\lfloor\frac{u-1}{2}\right\rfloor}(-1)^{u+k}\binom{u}{k}\left(h_{-uj+(u-2k)l}^{\Gamma}\right)\\
        &+\sum_{i=\left\lfloor\frac{u+3}{2}\right\rfloor}^{u-1}\sum_{k=0}^{\left\lfloor\frac{u-1}{2}\right\rfloor}(-1)^{k+i}\binom{u}{k}\left(\binom{u-1}{i-1}+\binom{u-1}{i}\right)\left(h_{(u-2i)j+(u-2k)l}^{\Gamma}\right)\\
        &+(u\mod2)\sum_{k=0}^{\frac{u-1}{2}}(-1)^{\frac{u-1}{2}+k}\binom{u}{\frac{u-1}{2}}\binom{u}{k}\left(h_{j+(u-2k)l}^{\Gamma}\right)\\
        &+(u\mod2)\sum_{k=0}^{\frac{u-1}{2}}(-1)^{\frac{u+1}{2}+k}\binom{u}{\frac{u+1}{2}}\binom{u}{k}\left(h_{-j+(u-2k)l}^{\Gamma}\right)\\
        &=\sum_{i=0}^{\left\lfloor\frac{u-2}{2}\right\rfloor}\sum_{k=0}^{\left\lfloor\frac{u}{2}\right\rfloor}(-1)^{k+i}\binom{u}{k}\binom{u}{i}\left(h_{(u-2i)j+(u-2k)l}^{\Gamma}\right)+((u+1)\mod2)\sum_{k=0}^{\frac{u-2}{2}}(-1)^{k+\frac{u}{2}}\binom{u}{k}\binom{u}{\frac{u}{2}}\left(h_{(u-2k)l}^{\Gamma}\right)\\
        &+((u+1)\mod2)(-1)^{u}\binom{u}{\frac{u}{2}}\binom{u-1}{\frac{u-2}{2}}\left(h_{0}^{\Gamma}\right)+\sum_{k=0}^{\left\lfloor\frac{u-1}{2}\right\rfloor}(-1)^{u+k}\binom{u}{k}\left(h_{-uj+(u-2k)l}^{\Gamma}\right)\\
        &+\sum_{i=\left\lfloor\frac{u+3}{2}\right\rfloor}^{u-1}\sum_{k=0}^{\left\lfloor\frac{u-1}{2}\right\rfloor}(-1)^{k+i}\binom{u}{k}\binom{u}{i}\left(h_{(u-2i)j+(u-2k)l}^{\Gamma}\right)\\
        &+(u\mod2)\sum_{k=0}^{\frac{u-1}{2}}\sum_{i=\frac{u-1}{2}}^{\frac{u+1}{2}}(-1)^{i+k}\binom{u}{i}\binom{u}{k}\left(h_{(u-2i)j+(u-2k)l}^{\Gamma}\right)\\
        &=\sum_{i=0}^{\left\lfloor\frac{u-2}{2}\right\rfloor}\sum_{k=0}^{\left\lfloor\frac{u-1}{2}\right\rfloor}(-1)^{k+i}\binom{u}{k}\binom{u}{i}\left(h_{(u-2i)j+(u-2k)l}^{\Gamma}\right)+((u+1)\mod2)\sum_{i=0}^{\frac{u-2}{2}}(-1)^{\frac{u}{2}+i}\binom{u}{\frac{u}{2}}\binom{u}{i}\left(h_{(u-2i)j}^{\Gamma}\right)\\
        &+((u+1)\mod2)\sum_{k=0}^{\frac{u-2}{2}}(-1)^{k+\frac{u}{2}}\binom{u}{k}\binom{u}{\frac{u}{2}}\left(h_{(u-2k)l}^{\Gamma}\right)+((u+1)\mod2)(-1)^{u}\binom{u}{\frac{u}{2}}\binom{u-1}{\frac{u-2}{2}}\left(h_{0}^{\Gamma}\right)\\
        &+\sum_{i=\left\lfloor\frac{u+3}{2}\right\rfloor}^{u}\sum_{k=0}^{\left\lfloor\frac{u-1}{2}\right\rfloor}(-1)^{k+i}\binom{u}{k}\binom{u}{i}\left(h_{(u-2i)j+(u-2k)l}^{\Gamma}\right)\\
        &+(u\mod2)\sum_{k=0}^{\frac{u-1}{2}}\sum_{i=\frac{u-1}{2}}^{\frac{u+1}{2}}(-1)^{i+k}\binom{u}{i}\binom{u}{k}\left(h_{(u-2i)j+(u-2k)l}^{\Gamma}\right)\\
        &=\sum_{i=0}^{u}\sum_{k=0}^{\left\lfloor\frac{u-1}{2}\right\rfloor}(-1)^{k+i}\binom{u}{k}\binom{u}{i}\left(h_{(u-2i)j+(u-2k)l}^{\Gamma}\right)\\
        &+2((u+1)\mod2)\sum_{i=0}^{\frac{u-2}{2}}(-1)^{\frac{u}{2}+i}\binom{u-1}{\frac{u-2}{2}}\binom{u}{i}\left(h_{(u-2i)j}^{\Gamma}\right)\\
        &+((u+1)\mod2)(-1)^{u}\binom{u}{\frac{u}{2}}\binom{u-1}{\frac{u-2}{2}}\left(h_{0}^{\Gamma}\right)\\
        &=\sum_{i=0}^{u}\sum_{k=0}^{\left\lfloor\frac{u-1}{2}\right\rfloor}(-1)^{k+i}\binom{u}{k}\binom{u}{i}\left(h_{(u-2i)j+(u-2k)l}^{\Gamma}\right)\\
        &+((u+1)\mod2)\sum_{i=0}^{\frac{u-2}{2}}(-1)^{\frac{u}{2}+i}\binom{u-1}{\frac{u-2}{2}}\binom{u}{i}\left(h_{(u-2i)j}^{\Gamma}\right)\\
        &+((u+1)\mod2)\sum_{i=0}^{\frac{u-2}{2}}(-1)^{\frac{u}{2}+i}\binom{u-1}{\frac{u-2}{2}}\binom{u}{i}\left(h_{(u-2i)j}^{\Gamma}\right)\\
        &+((u+1)\mod2)(-1)^{u}\binom{u}{\frac{u}{2}}\binom{u-1}{\frac{u-2}{2}}\left(h_{0}^{\Gamma}\right)\\
        &=\sum_{i=0}^{u}\sum_{k=0}^{\left\lfloor\frac{u-1}{2}\right\rfloor}(-1)^{k+i}\binom{u}{k}\binom{u}{i}\left(h_{(u-2i)j+(u-2k)l}^{\Gamma}\right)\\
        &+((u+1)\mod2)\sum_{i=0}^{\frac{u}{2}}(-1)^{\frac{u}{2}+i}\binom{u-1}{\frac{u-2}{2}}\binom{u}{i}\left(h_{(u-2i)j}^{\Gamma}\right)\\
        &+((u+1)\mod2)\sum_{i=0}^{\frac{u-2}{2}}(-1)^{\frac{u}{2}+i}\binom{u-1}{\frac{u-2}{2}}\binom{u}{i}\left(h_{-(u-2i)j}^{\Gamma}\right)\\
        &=\sum_{i=0}^{u}\sum_{k=0}^{\left\lfloor\frac{u-1}{2}\right\rfloor}(-1)^{k+i}\binom{u}{k}\binom{u}{i}\left(h_{(u-2i)j+(u-2k)l}^{\Gamma}\right)\\
        &+((u+1)\mod2)\sum_{i=0}^{\frac{u}{2}}(-1)^{\frac{u}{2}+i}\binom{u-1}{\frac{u-2}{2}}\binom{u}{i}\left(h_{(u-2i)j}^{\Gamma}\right)\\
        &+((u+1)\mod2)\sum_{i=\frac{u+2}{2}}^{u}(-1)^{\frac{3u}{2}-i}\binom{u-1}{\frac{u-2}{2}}\binom{u}{i}\left(h_{(u-2i)j}^{\Gamma}\right)\\
        &=\sum_{k=0}^{\left\lfloor\frac{u-1}{2}\right\rfloor}\sum_{i=0}^{u}(-1)^{k+i}\binom{u}{k}\binom{u}{i}\left(h_{(u-2i)j+(u-2k)l}^{\Gamma}\right)\\
        &+((u+1)\mod2)\sum_{i=0}^{u}(-1)^{\frac{u}{2}+i}\binom{u-1}{\frac{u-2}{2}}\binom{u}{i}\left(h_{(u-2i)j}^{\Gamma}\right)
    \end{align*}
\hfill\qedsymbol

    \subsection*{Proof of \eqref{p2n+1} and \eqref{p2n} from Lemma \ref{LambdaLambda}}
    
    Recall that the statement we wish to prove is for all $n\in{\mathbb Z}_+$
    \begin{align*}
        p_{2n+1}^\Gamma(j,l)
        &=\sum_{i=0}^{n}\sum_{k=0}^{n}(-1)^{k+i+1}\binom{2n+1}{i}\binom{2n+1}{k}\Lambda^\Gamma_{(2n+1-2i)j,(2n+1-2k)l,1}.\\
        p^\Gamma_{2n}(j,l)
        &=\sum_{i=0}^{n-1}\sum_{k=0}^{2n}(-1)^{k+i+1}\binom{2n-1}{i}\binom{2n}{k}\Lambda^\Gamma_{(2n-1-2i)j+(2n-2k)l,j,1}.
    \end{align*}
    
    For \eqref{p2n+1}, by Proposition \ref{p_u}, we have:
    \begin{align*}
        p_{2n+1}^\Gamma(j,l)
        &=\sum_{k=0}^{n}\sum_{i=0}^{2n+1}(-1)^{k+i}\binom{2n+1}{i}\binom{2n+1}{k}h^\Gamma_{(2n+1-2i)j+(2n+1-2k)l}\\
        &=\sum_{k=0}^{n}\sum_{i=0}^{n}(-1)^{k+i}\binom{2n+1}{i}\binom{2n+1}{k}h^\Gamma_{(2n+1-2i)j+(2n+1-2k)l}\\
        &+\sum_{k=0}^{n}\sum_{i=n+1}^{2n+1}(-1)^{k+i}\binom{2n+1}{i}\binom{2n+1}{k}h^\Gamma_{(2n+1-2i)j+(2n+1-2k)l}\\
        &=-\sum_{k=0}^{n}\sum_{i=0}^{n}(-1)^{k+i+1}\binom{2n+1}{i}\binom{2n+1}{k}h^\Gamma_{(2n+1-2i)j+(2n+1-2k)l}\\
        &+\sum_{k=0}^{n}\sum_{i=n+1}^{2n+1}(-1)^{k+i}\binom{2n+1}{i}\binom{2n+1}{k}h^\Gamma_{-(2n+1-2i)j-(2n+1-2k)l}\\
        &=-\sum_{k=0}^{n}\sum_{i=0}^{n}(-1)^{k+i+1}\binom{2n+1}{i}\binom{2n+1}{k}h^\Gamma_{(2n+1-2i)j+(2n+1-2k)l}\\
        &+\sum_{k=0}^{n}\sum_{i=0}^{n}(-1)^{k+i+1}\binom{2n+1}{i}\binom{2n+1}{k}h^\Gamma_{(2n+1-2i)j-(2n+1-2k)l}\\
        &=\sum_{k=0}^{n}\sum_{i=0}^{n}(-1)^{k+i+1}\binom{2n+1}{i}\binom{2n+1}{k}\Lambda^\Gamma_{(2n+1-2i)j,(2n+1-2k)l,1}.
    \end{align*}
    To prove \eqref{p2n} we use Proposition \ref{p_u} again to get
    \begin{align*}
        p_{2n}^\Gamma(j,l)
        &=\sum_{k=0}^{n-1}\sum_{i=0}^{2n}(-1)^{k+i}\binom{2n}{i}\binom{2n}{k}h^\Gamma_{(2n-2i)j+(2n-2k)l}+\sum_{i=0}^{2n}(-1)^{n+i}\binom{2n-1}{n-1}\binom{2n}{i}h^\Gamma_{(2n-2i)j}\\
        &=\sum_{k=0}^{n-1}\sum_{i=0}^{n-1}(-1)^{k+i}\binom{2n}{i}\binom{2n}{k}h^\Gamma_{(2n-2i)j+(2n-2k)l}+\sum_{k=0}^{n-1}\sum_{i=n}^{2n}(-1)^{k+i}\binom{2n}{i}\binom{2n}{k}h^\Gamma_{(2n-2i)j+(2n-2k)l}\\
        &+\sum_{i=0}^{n-1}(-1)^{n+i}\binom{2n-1}{n-1}\binom{2n}{i}h^\Gamma_{(2n-2i)j}+\sum_{i=n}^{2n}(-1)^{n+i}\binom{2n-1}{n-1}\binom{2n}{i}h^\Gamma_{(2n-2i)j}\\
        &=\sum_{k=0}^{n-1}\sum_{i=0}^{n-1}(-1)^{k+i}\binom{2n}{i}\binom{2n}{k}h^\Gamma_{(2n-2i)j+(2n-2k)l}+\sum_{k=0}^{n-1}\sum_{i=n}^{2n}(-1)^{k+i}\binom{2n}{i}\binom{2n}{k}h^\Gamma_{-(2n-2i)j-(2n-2k)l}\\
        &+\sum_{i=0}^{n-1}(-1)^{n+i}\binom{2n-1}{n-1}\binom{2n}{i}h^\Gamma_{(2n-2i)j}+\sum_{i=n}^{2n}(-1)^{n+i}\binom{2n-1}{n-1}\binom{2n}{i}h^\Gamma_{-(2n-2i)j}\\
        &=\sum_{k=0}^{n-1}\sum_{i=0}^{n-1}(-1)^{k+i}\binom{2n}{i}\binom{2n}{k}h^\Gamma_{(2n-2i)j+(2n-2k)l}+\sum_{k=0}^{n-1}\sum_{i=0}^{n}(-1)^{k+i}\binom{2n}{i}\binom{2n}{k}h^\Gamma_{(2n-2i)j-(2n-2k)l}\\
        &+\sum_{i=0}^{n-1}(-1)^{n+i}\binom{2n-1}{n-1}\binom{2n}{i}h^\Gamma_{(2n-2i)j}+\sum_{i=0}^{n}(-1)^{n+i}\binom{2n-1}{n-1}\binom{2n}{i}h^\Gamma_{(2n-2i)j}\\
        &=\sum_{k=0}^{n-1}\sum_{i=0}^{n-1}(-1)^{k+i}\binom{2n}{i}\binom{2n}{k}h^\Gamma_{(2n-2i)j+(2n-2k)l}+\sum_{k=n+1}^{2n}\sum_{i=0}^{n}(-1)^{k+i}\binom{2n}{i}\binom{2n}{k}h^\Gamma_{(2n-2i)j+(2n-2k)l}\\
        &+\sum_{i=0}^{n-1}(-1)^{n+i}\binom{2n-1}{n-1}\binom{2n}{i}h^\Gamma_{(2n-2i)j}+\sum_{i=0}^{n}(-1)^{n+i}\binom{2n-1}{n-1}\binom{2n}{i}h^\Gamma_{(2n-2i)j}\\
        &=\sum_{k=0}^{n-1}\sum_{i=0}^{n-1}(-1)^{k+i}\binom{2n}{i}\binom{2n}{k}h^\Gamma_{(2n-2i)j+(2n-2k)l}+\sum_{k=n+1}^{2n}\sum_{i=0}^{n-1}(-1)^{k+i}\binom{2n}{i}\binom{2n}{k}h^\Gamma_{(2n-2i)j+(2n-2k)l}\\
        &+\sum_{k=n+1}^{2n}(-1)^{k+n}\binom{2n}{n}\binom{2n}{k}h^\Gamma_{(2n-2k)l}+2\sum_{i=0}^{n-1}(-1)^{n+i}\binom{2n-1}{n-1}\binom{2n}{i}h^\Gamma_{(2n-2i)j}+\binom{2n-1}{n-1}\binom{2n}{n}h^\Gamma_{0}\\
        &=\sum_{k=0}^{n-1}\sum_{i=0}^{n-1}(-1)^{k+i}\binom{2n}{i}\binom{2n}{k}h^\Gamma_{(2n-2i)j+(2n-2k)l}+\sum_{k=n+1}^{2n}\sum_{i=0}^{n-1}(-1)^{k+i}\binom{2n}{i}\binom{2n}{k}h^\Gamma_{(2n-2i)j+(2n-2k)l}\\
        &+2\sum_{k=n+1}^{2n}(-1)^{k+n}\binom{2n-1}{n-1}\binom{2n}{k}h^\Gamma_{(2n-2k)l}+\sum_{i=0}^{n-1}(-1)^{n+i}\binom{2n}{n}\binom{2n}{i}h^\Gamma_{(2n-2i)j}+\binom{2n-1}{n-1}\binom{2n}{n}h^\Gamma_{0}\\
        &=\sum_{k=0}^{2n}\sum_{i=0}^{n-1}(-1)^{k+i}\binom{2n}{i}\binom{2n}{k}h^\Gamma_{(2n-2i)j+(2n-2k)l}+\sum_{k=n+1}^{2n}(-1)^{k+n}\binom{2n-1}{n-1}\binom{2n}{k}h^\Gamma_{(2n-2k)l}\\
        &+\sum_{k=n+1}^{2n}(-1)^{k+n}\binom{2n-1}{n-1}\binom{2n}{k}h^\Gamma_{(2n-2k)l}+\binom{2n-1}{n-1}\binom{2n}{n}h^\Gamma_{0}\\
        &=\sum_{k=0}^{2n}\sum_{i=0}^{n-1}(-1)^{k+i}\binom{2n}{i}\binom{2n}{k}h^\Gamma_{(2n-2i)j+(2n-2k)l}+\sum_{k=n+1}^{2n}(-1)^{k+n}\binom{2n-1}{n-1}\binom{2n}{k}h^\Gamma_{-(2n-2k)l}\\
        &+\sum_{k=n+1}^{2n}(-1)^{k+n}\binom{2n-1}{n-1}\binom{2n}{k}h^\Gamma_{(2n-2k)l}+\binom{2n-1}{n-1}\binom{2n}{n}h^\Gamma_{0}\\
        &=\sum_{k=0}^{2n}\sum_{i=0}^{n-1}(-1)^{k+i}\binom{2n}{i}\binom{2n}{k}h^\Gamma_{(2n-2i)j+(2n-2k)l}+\sum_{k=0}^{n-1}(-1)^{k+n}\binom{2n-1}{n-1}\binom{2n}{k}h^\Gamma_{(2n-2k)l}\\
        &+\sum_{k=n+1}^{2n}(-1)^{k+n}\binom{2n-1}{n-1}\binom{2n}{k}h^\Gamma_{(2n-2k)l}+\binom{2n-1}{n-1}\binom{2n}{n}h^\Gamma_{0}\\
        &=\sum_{k=0}^{2n}\sum_{i=0}^{n-1}(-1)^{k+i}\binom{2n}{i}\binom{2n}{k}h^\Gamma_{(2n-2i)j+(2n-2k)l}+\sum_{k=0}^{2n}(-1)^{k+n}\binom{2n-1}{n-1}\binom{2n}{k}h^\Gamma_{(2n-2k)l}\\
        &=\sum_{i=0}^{n-1}\sum_{k=0}^{2n}(-1)^{k+i}\left(\binom{2n-1}{i}+\binom{2n-1}{i-1}\right)\binom{2n}{k}h^\Gamma_{(2n-2i)j+(2n-2k)l}\\
        &+\sum_{k=0}^{2n}(-1)^{k+n}\binom{2n-1}{n-1}\binom{2n}{k}h^\Gamma_{(2n-2k)l}\textnormal{ by Pascal's Rule}\\
        &=\sum_{i=0}^{n-1}\sum_{k=0}^{2n}(-1)^{k+i}\binom{2n-1}{i}\binom{2n}{k}h^\Gamma_{(2n-2i)j+(2n-2k)l}\\
        &+\sum_{i=1}^{n}\sum_{k=0}^{2n}(-1)^{k+i}\binom{2n-1}{i-1}\binom{2n}{k}h^\Gamma_{(2n-2i)j+(2n-2k)l}\\
        &=\sum_{i=0}^{n-1}\sum_{k=0}^{2n}(-1)^{k+i}\binom{2n-1}{i}\binom{2n}{k}h^\Gamma_{(2n-2i)j+(2n-2k)l}\\
        &+\sum_{i=0}^{n-1}\sum_{k=0}^{2n}(-1)^{k+i+1}\binom{2n-1}{i}\binom{2n}{k}h^\Gamma_{(2n-2-2i)j+(2n-2k)l}\\
        &=-\sum_{i=0}^{n-1}\sum_{k=0}^{2n}(-1)^{k+i+1}\binom{2n-1}{i}\binom{2n}{k}\left(h^\Gamma_{(2n-2i)j+(2n-2k)l}-h^\Gamma_{(2n-2-2i)j+(2n-2k)l}\right)\\
        &=\sum_{i=0}^{n-1}\sum_{k=0}^{2n}(-1)^{k+i+1}\binom{2n-1}{i}\binom{2n}{k}\Lambda^\Gamma_{(2n-1-2i)j+(2n-2k)l,j,1} \hfill\qedsymbol
            \end{align*}
            
            }
    

\begin{thebibliography}{99}

\bibitem{BC} I. Bagci, S. Chamberlin, \textit{Integral bases for the universal enveloping algebras of map superalgebras}, J. Pure Appl. Algebra \textbf{218} (8) (2014), pp. 1563--1576. DOI: https://doi.org/10.1007/s10468-016-9603-x 


\bibitem{BT07} G. Benkart and P. Terwilliger, \textit{The universal central extension of the three-point   {$\mathfrak{sl}_2$}-loop algebra}, Proc. Amer. Math. Soc., \textbf{135}(6) (2007), pp. 1659--1668. DOI: https://doi.org/10.1090/S0002-9939-07-08765-5


\bibitem{BiC1} A. Bianchi, S. Chamberlin, \textit{Finite-dimensional representations of hyper multicurrent and multiloop algebras}, Algebr Represent Theor (2020). DOI: https://doi.org/10.1007/s10468-020-09955-z

\bibitem{BiC2} A. Bianchi, S. Chamberlin, \textit{Weyl  Modules  and  Weyl  Functors  for  hyper--map algebras} (in preparation).

\bibitem{BM} A. Bianchi, A. Moura,\textit{ Finite-dimensional representations of twisted hyper-loop algebras}, Comm. Algebra \textbf{42} (7) (2014), pp. 3147--3182. DOI: https://doi.org/10.1080/00927872.2013.781610


\bibitem{C} S. Chamberlin, \textit{Integral bases for the universal enveloping algebras of map algebras}, J. of Algebra 377 (1) (2013),  pp. 232--249. DOI: https://doi.org/10.1016/j.jalgebra.2012.11.046

\bibitem{CP} V. Chari, A. Pressley, \textit{Quantum affine algebras at roots of unity}, Represent. Theory \textbf{1} (1997), pp. 280–-328. DOI: https://doi.org/10.1090/S1088-4165-97-00030-7


\bibitem{el} C. EI-Chaar, \textit{The Onsager algebra}, Master thesis at University of Ottawa, (2010). DOI: http://dx.doi.org/10.20381/ruor-19393

\bibitem{DR00} E. Date, S. Roan, \textit{The algebraic structure of the Onsager algebra},  Czechoslovak J. Phys., \textbf{50}(1) (2000), pp. 37--44. DOI: https://doi.org/10.1023/A:1022812728907

\bibitem{E07} A. Elduque, \textit{The {$S\sb 4$}-action on tetrahedron algebra}, Proc. Roy. Soc. Edinburgh Sect. A, \textbf{137}(6) 2007, pp. 1227--1248. DOI: https://doi.org/10.1017/S0308210506000473


\bibitem{G} H. Garland, \textit{The arithmetic theory of loop algebras}, J. Algebra \textbf{53} (1978) , pp.480--551. DOI: https://doi.org/10.1016/0021-8693(78)90294-6

\bibitem{HT07} B. Hartwig, P.Terwilliger, \textit{ The tetrahedron algebra, the Onsager algebra, and the {$\mathfrak{sl}\sb 2$}-loop algebra}, J. Algebra \textbf{308}(2) (2007), pp. 840--863. DOI: https://doi.org/10.1016/j.jalgebra.2006.09.011

\bibitem{H} J. E. Humphreys, Introduction to Lie algebras and representation theory, Springer--Verlag, GTM \textbf{9} (1970). DOI: 10.1007/978-1-4612-6398-2

\bibitem{JM1} D. Jakelic, A. Moura, \textit{Finite-dimensional representations of hyper loop algebras}, Pacific J. Math. \textbf{233} (2) (2007), pp. 371--402. DOI: 10.2140/pjm.2007.233.371


\bibitem{K} B. Kostant, \textit{Groups over $\mathbb Z$}, Algebraic Groups and Discontinuous Subgroups, Proc. Symp. Pure Math. IX, Providence, AMS (1966). DOI: https://doi.org/10.1007/b94535\_21


\bibitem{M} D. Mitzman, \textit{Integral Bases for Affine Lie Algebras and Their Universal Enveloping Algebras}, Contemp. Math.\textbf{ 40} (1983). DOI: http://dx.doi.org/10.1090/conm/040

\bibitem{NSS} E. Neher, A. Savage, P. Senesi, \textit{Irreducible finite-dimensional representations of equivariant map algebras}, Trans. Amer. Math. Soc. \textbf{364} (5) (2012), pp.  2619--2646. DOI: https://doi.org/10.1090/S0002-9947-2011-05420-6

\bibitem{NSS2} E. Neher, A. Savage, \textit{A survey of equivariant map algebras with open problems}, Contemp. Math.\textbf{ 602} (2013). DOI: http://dx.doi.org/10.1090/conm/602

\bibitem{O44} L. Onsager, \textit{Crystal statistics. i. a two-dimensional model with an order-disorder transition}, Phys. Rev.  \textbf{65} (3-4) (1944), pp. 117--149. DOI: https://doi.org/10.1103/PhysRev.65.117

\bibitem{R91} S. Roan, \textit{Onsager's algebra, loop algebra and chiral potts model}, Preprint Max-Planck-Inst. f\"ur Math., Bonn, MPI 91-70, (1991).

\end{thebibliography}
        \end{document}